\documentclass[11pt,reqno]{amsart}

\usepackage{amsmath, amsfonts, amssymb,  mathrsfs,  array, stmaryrd,  indentfirst, amsthm,  hyperref, comment}
\usepackage{graphicx, enumitem, tabularx}

\usepackage[margin=1.0in]{geometry}
\usepackage{setspace}
\numberwithin{equation}{section}
\theoremstyle{plain}
\newtheorem{lemma}{Lemma}[section]
\newtheorem{thm}{Theorem}[section]
\newtheorem{defn}{Definition}[section]
\newtheorem{prop}{Proposition}[section]
\newtheorem{cor}{Corollary}[section]

\newtheorem{remark}{Remark}[section]

\newcommand{\lbar}[1]{\underline{#1}}
\newcommand{\ubar}[1]{\overline{#1}}

\newcommand{\bfpi}{\boldsymbol\pi}
\newcommand{\bftheta}{\boldsymbol\theta}
\newcommand{\bfvarpi}{\boldsymbol\varpi}
\newcommand{\bfvartheta}{\boldsymbol\vartheta}

\makeatletter
\def\@setcopyright{}
\def\serieslogo@{}
\makeatother

\begin{document}

\title[]{Minimizing the Probability of Lifetime Ruin Under Ambiguity Aversion}\thanks{This research is supported by the National Science Foundation under grant  DMS-0955463.}

\author{Erhan Bayraktar}
\address[Erhan Bayraktar]{Department of Mathematics, University of Michigan, 530 Church Street, Ann Arbor, MI 48109, USA}
\email{erhan@umich.edu}
\author{Yuchong Zhang }
\address[Yuchong Zhang]{Department of Mathematics, University of Michigan, 530 Church Street, Ann Arbor, MI 48109, USA}
\email{yuchong@umich.edu}

\begin{abstract}
We determine the optimal robust investment strategy of an individual who targets at a given rate of consumption and seeks to minimize the probability of lifetime ruin when she does not have perfect confidence in the drift of the risky asset. Using stochastic control, we characterize the value function as the unique classical solution of an associated Hamilton-Jacobi-Bellman (HJB) equation, obtain feedback forms for the optimal investment and drift distortion, and discuss their dependence on various model parameters. In analyzing the HJB equation, we establish the existence and uniqueness of viscosity solution using Perron's method, and then upgrade regularity by working with an equivalent convex problem obtained via the Cole-Hopf transformation. We show the original value function may lose convexity for a class of parameters and the Isaacs condition may fail. Numerical examples are also included to illustrate our results.
\end{abstract}
\keywords{Probability of lifetime ruin, ambiguity aversion, drift uncertainty, viscosity solutions, Perron's method, regularity.}

\maketitle
\pagestyle{headings}

\section{Introduction}
The problem of how individuals should invest their wealth in a risky financial market to minimize the probability that they outlive their wealth, also known as the \textit{probability of lifetime ruin} (this terms was coined by  \cite{MilevskyRobinson00}), was analyzed by Young \cite{Young04}. We mention that Jacka in an earlier work \cite{Jacka02} considered a finite-fuel problem of very similar form.
Subsequent variants of Young's work include but not limited to adding borrowing constraints \cite{BayraktarYoung07b}, assuming consumption is ratcheted \cite{BayraktarYoung08}, allowing stochastic consumption \cite{BayraktarYoung11} and stochastic volatility \cite{Xueying11}. In all previous works, there is a fixed risky asset model; that is, the investor is certain about the evolution and distribution of the risky asset price. This is, however, not very realistic. There may be good estimates of the price volatility, but drift estimation, as Rogers points out in \cite[Section 4.2]{Rogers}, is almost impossible; it would require centuries of data to obtain a reliable estimate. Therefore, it is desirable to have a robust investment strategy that can perform well against drift misspecification. For a good introduction of robust decision making theory, see \cite{Hansen}.

Although drift estimation is difficult, one would still like to make use of the available data. A natural approach is to extract from the available data a reference model, and penalize other models based on their deviation from the reference model. How hard to penalize depends on how averse the agent is to \textit{ambiguity}, also called \textit{model uncertainty} or \textit{Knightian uncertainty}. Early works incorporating ambiguity aversion into optimization (e.g. \cite{Maenhout04}, \cite{Hansen06}) 
are mostly done via a formal analysis of the corresponding Hamilton-Jacobi-Bellman (HJB) equation.
Among those that provide more mathematical rigor, we mention a few that use different approaches. Jaimungal solves a finite horizon irreversible investment problem \cite{Jaimungal11}, and a hybrid model of default problem with Sigloch \cite{Jaimungal12} using stochastic control. They work with a scaled entropic penalty in order to get explicit solutions and rely on direct verification. Bordigoni et al. \cite{Bordigoni07} analyze a finite horizon utility maximization problem also by control method, but provide a backward stochastic differential equation (BSDE) characterization instead of an HJB characterization. Their results are generalized to an infinite horizon setting by Hu and Schweizer \cite{Hu11}. Schied \cite{Schied07} and Hern\'{a}ndez-Hern\'{a}ndez and Schied \cite{HSchied07} treat robust utility maximization problems using duality or a combination of duality and control.

In this paper, we provide a complete and rigorous analysis of the robust lifetime ruin problem 
\[\inf_\pi \sup_\mathbb{Q} \left\{\mathbb{Q}(\tau_b<\tau_d)-\frac{1}{\varepsilon}h^d(\mathbb{Q}|\mathbb{P})\right\}\]
using stochastic control, where $\tau_b$ and $\tau_d$ are the ruin time and death time, respectively, $h^d$ is a variant of the entropic penalty function which only measures entropy up to the death time, $\varepsilon$ specifies the penalization strength, $\pi$ runs through a set of investment strategies and $\mathbb{Q}$ runs through a set of possible models representing drift uncertainty.
When the hazard rate is zero, we obtain explicit formulas. In the general case, we characterize the value function as the unique classical solution of an associated HJB equation satisfying two boundary conditions, and give feedback forms for the optimal investment and drift distortion. In contrast to the non-robust case or robust utility maximization problem, we show that the value function loses convexity for a class of parameters, which suggests that the Isaacs condition may fail. Same as the non-robust case, we also show that the optimally controlled wealth process never reaches the so-called ``safe level". This is different from the zero-hazard rate case,  goes back to the work of Pestien and Sudderth \cite{Sudderth85} (also see \cite{Bauerle13}). The goal in the zero-hazard rate case is to reach the safe level (possibly in infinite time) because the individual never dies, whereas the goal when the hazard rate is non-zero is to stay away from the ruin level and to ``win" the game by dying.  Without a deadline, the optimal strategy is to maximize the ratio of drift to volatility squared. Adding death to the problem changes it tremendously. In particular, we shall see that in terms of the optimal investment strategy, the robustness is only non-trivial then. 

Our work extends the discussion in \cite{Young04} to the robust case. 
Unlike \cite{Jaimungal11} and \cite{Jaimungal12} where a scaled entropic penalty leads to explicit solutions, our random horizon robust problem, even in the simple Black-Scholes framework, fails to have an explicit solution in general, whether the penalty is scaled or not. Moreover, due to degeneracy and the control space being unbounded, the classical nonlinear elliptic theory by Krylov \cite{Krylov} cannot be applied directly. So we have to resort to the theory of viscosity solutions and then upgrade regularity by bootstrapping. Our work differs from \cite{Bordigoni07}, \cite{Hu11}, \cite{Schied07}, \cite{HSchied07} in the methodology. The BSDE characterizations in \cite{Bordigoni07} and \cite{Hu11} only focus on the inner $\mathbb{Q}$-maximization problem and do not describe the optimal investment strategy or the saddle point. The duality approach of \cite{Schied07} requires the infimum and supremum to be exchangeable, which does not hold in our case with certain choice of parameters. The classical duality $\log\mathbb{E}[e^{X}]=\sup_{\mathbb{Q}\in\mathcal{Q}_{abs}}\{\mathbb{E}^{\mathbb{Q}}[X]-h(\mathbb{Q}|\mathbb{P})\}$ between free energy and entropic penalty may look useful at a first glance, but the uncertainty set $\mathcal{Q}_{abs}$ does not preserve the independence between asset price and mortality, and does not leave room for varied confidence levels regarding different model components.\footnote{$\mathcal{Q}_{abs}$ denotes the set of measures that are absolutely continuous with respect to $\mathbb{P}$ and have finite entropy.} In addition, we are not using the exact entropic function $h$, but its variant $h^d$. Due to time-inconsistency issue, we do not consider uncertainty in hazard rate. It could be an interesting extension to have uncertain Poisson jump rate (see e.g. \cite{LimShanthikumar}, \cite{MataramvuraOksendal}, \cite{Jaimungal13}, \cite{BenBertBrown}) and to allow varied positive levels of ambiguity aversion (see e.g. \cite{Wang03}, \cite{Jaimungal11}).

For the construction of a viscosity solution to the HJB equation, we use a ``comparison + Perron's method" approach described in \cite{UsersGuide} instead of the usual route of ``dynamic programming principle (DPP) + value function is a viscosity solution + comparison". The reason is that robust optimization problems resemble stochastic differential games in which nature can be regarded as the second player, and the DPP for games is generally complicated because of measurability issues. One either has to use the Elliott-Kalton formulation where one player uses controls and the other player uses ``strategies", i.e. maps defined on a set of controls satisfying nonanticipitivity (see e.g. \cite{Fleming89}, \cite{Fleming11}, \cite{BayraktarSong}), or restrict oneself to strategies of simple form, for example, to what S{\^{\i}}rbu \cite{Sirbu13} calls elementary strategies. Both ways to get around the measurability issues are not ideal for us. In particular, it is a bit unnatural for us to use the Elliott-Kalton formulation and assume nature is a strategic player against us, because nature has no payoff and is disinterested. It turns out that the classical Perron's method yields a much simpler and more elegant construction. The only drawback is that regularity now becomes very important, otherwise the constructed solution cannot be related to the value function. Fortunately, we are able to upgrade regularity and carry out a verification theorem. The approach outlined here was first used by Jane\v{c}ek and S{\^{\i}}rbu \cite{Janecek12} in a pure stochastic control problem.

Convexity is usually key to upgrading regularity. One challenge introduced by robustness, as we have pointed out, is the loss of convexity of the value function for a class of parameters. In fact, even for non-robust lifetime ruin problems, \textit{a priori} convexity of the value function is not clear. For example, \cite{BayraktarYoung11} obtains convexity for a lifetime ruin problem with stochastic consumption by going to a controller-and-stopper problem whose convex dual is related to the original problem through a dimension reduction. We overcome this challenge by working with an equivalent convex problem obtained through the Cole-Hopf transformation. Once we have convexity, it is easy to upgrade to $C^1$-regularity using convex analysis and the theory of viscosity solutions. We further upgrade to $C^2$-regularity by analyzing a Poisson equation, where we borrow some techniques from \cite{Janecek12} and \cite{ShreveSoner94}. One may try to prove $C^2$-regularity by the regularization method used in \cite{Zariphopoulou94} and \cite{Duffie97}, 
but such an approach requires us to prove the existence of a positive lower bound on $\pi$ that is independent of the regularization on compact intervals away from the safe level, which we find to be difficult to establish.

The rest of the paper is organized as follows. In Section~\ref{sec:formulation}, we set up the problem, derive the HJB equation and feedback forms heuristically, and state the main results. Section~\ref{sec:psi0} provides an explicit solution when the hazard rate is zero, which is not only interesting for its own sake, but serves as a useful upper bound in the analysis of the general case. Sections~\ref{sec:viscosity_soln} and ~\ref{sec:regularity} are devoted to establishing the existence of a classical solution to the HJB equation, with Secitons~\ref{sec:viscosity_soln} focusing on Perron's construction of a viscosity solution, and Section~\ref{sec:regularity} on regularity. In Section~\ref{sec:verification}, we give a verification theorem and the proof of our main results. In order to prove verification theorem, we also show the boundedness and Lipschitz continuity of the optimal investment strategy. Sections~\ref{sec:properties} collects some additional properties of the optimal investment strategy and the value function. Sections~\ref{sec:numerics} provides numerical results and formulas for small $\varepsilon$-expansion.

\section{Problem Formulation and Main Results}\label{sec:formulation}
Let $\Omega^M$ be the space of continuous functions $\omega: [0,\infty)\rightarrow \mathbb{R}$, equipped with the topology of uniform convergence on compact subintervals of $[0,\infty)$. Let $\mathcal{F}^M$ be the Borel sigma-algebra on $\Omega^M$ and $\mathbb{P}^M$ be the Wiener measure on $(\Omega^M,\mathcal{F}^M)$. The coordinate map $B_t(\omega):=\omega(t)$ is a standard Brownian motion in this space. Here $\mathbb{P}^M$ serves as a reference measure which reflects an individual's belief about the market. Let $N=(N_t)_{t\geq 0}$ be a Poisson process with rate $\lambda$ defined on another probability space $(\Omega^d,\mathcal{F}^d,\mathbb{P}^d)$. Let $\tau_d$ be the first time that the Poisson process jumps, modeling the death time of the individual. $\tau_d$ is an exponential random variable with parameter $\lambda$ which is known as the \textit{hazard rate} in this context. Define
\[(\Omega,\mathcal{F},\mathbb{P}):=(\Omega^M\times\Omega^d, \mathcal{F}^M\otimes \mathcal{F}^d, \mathbb{P}^M\times\mathbb{P}^d).\]
$B$ and $N$ are independent on this space, and remain a Brownian motion and a Poisson process, respectively. Let $\mathbb{F}=(\mathcal{F}_t)_{t\geq 0}$ be the (raw) filtration generated by the Brownian motion $B$ and $\mathbb{G}=(\mathcal{G}_t)_{t\geq 0}$ be the filtration generated by $B$ and the process $1_{\{\tau_d\leq t\}}$. Assume both $\mathbb{F}$ and $\mathbb{G}$ have been made right continuous. However, we do not complete the filtrations because later on, we would like to include measures that are only locally equivalent to $\mathbb{P}$ as part of our consideration.\footnote{By locally equivalent, we mean equivalent on $\mathcal{G}_t$ for all $t\geq 0$. Although the filtrations in our setup is not complete, stochastic integral can still be defined and has all the usual properties. In particular, It\^{o}'s lemma is still valid. See, for example, chapter 1 of \cite{Jacod}.}

The individual invests in a financial market consists of a risk-free bank account with interest rate $r>0$ and a risky asset whose price $S_t$ follows a geometric Brownian motion:
\begin{equation*}
dS_t=\mu S_t+\sigma S_t dB_t, \quad S_0=S>0
\end{equation*}
where $\mu>r$ and $\sigma>0$. Let $\bfpi_t$ be the amount that the individual invests in the risky asset at time $t$. Apart from investment, the individual also consumes at a constant rate $c>0$ of her current wealth $w$.\footnote{To simplify the discussion, we only work with constant consumption rate. But the main techniques can be applied to proportional consumption rate, and more generally, to the case when the consumption rate is a non-negative, Lipschitz continuous function of wealth.}
Her wealth $W_t$ evolves according to the stochastic differential equation (SDE):
\begin{equation*}
dW_t=[rW_t+(\mu-r)\bfpi_t-c]dt+\sigma\bfpi_t dB_t, \quad W_0=w.
\end{equation*}
An investment strategy $\bfpi$ is admissible if it is $\mathbb{F}$-progressively measurable and 
almost surely bounded (uniformly in time).\footnote{Almost sure boundedness can be relaxed as long as the best drift distortion in response to each $\bfpi$ defines an admissible measure $\mathbb{Q}\in\mathscr{Q}$ where $\mathscr{Q}$ is the model uncertainty set to be introduced.} 
Denote by $\mathscr{A}$ the set of all admissible strategies.

Let $\tau_b:=\inf\{t\geq 0: W_t\leq b\}$ be the first time the individual's wealth falls to or below a specified ruin level $b$. The individual aims at minimizing the probability that ruin happens before death, i.e. $\tau_b<\tau_d$, in a robust sense. More precisely, she suspects that the drift of the risky asset may be misspecified. 
So instead of optimizing under the reference measure $\mathbb{P}$, she considers a set $\mathscr{Q}$ of candidate measures that are locally equivalent to $\mathbb{P}$, and penalizes their deviation from $\mathbb{P}$. Here we assume the individual is only robust against the market model, but not the death time model, nor the independence between them. So elements in $\mathscr{Q}$ should be of the form $\mathbb{Q}^M\times \mathbb{P}^d$ so that $\tau_d$ remains an $exp(\lambda)$ random variable under all candidate measures. Let $h(\mathbb{Q}|\mathbb{P}):=\mathbb{E}^{\mathbb{Q}}[\log \frac{d\mathbb{Q}}{d\mathbb{P}}]$ be the relative entropic function. Denote by $\mathbb{Q}_{t}$ the restriction of a measure $\mathbb{Q}$ to $\mathcal{G}_t$. We penalize the deviation from $\mathbb{P}$ using a variant of $h$:
\begin{equation*}
h^d(\mathbb{Q}|\mathbb{P}):=h(\mathbb{Q}_{\tau_d}|\mathbb{P}_{\tau_d})
\end{equation*}
which only measures the relative entropy on $\mathcal{G}_{\tau_d}$; that is, the individual does not care about drift uncertainty after death. She faces the following robust optimization problem:
\begin{equation}\label{optprob}
\psi(w;\varepsilon)=\inf_{\pi\in\mathscr{A}}\sup_{\mathbb{Q}\in\mathscr{Q}} \left\{\mathbb{Q}_w(\tau_b<\tau_d)-\frac{1}{\varepsilon} h^d(\mathbb{Q}|\mathbb{P})\right\},
\end{equation}
where the subscript $w$ represents conditioning on the event $W_0=w$. The parameter $\varepsilon$ measures the individual's level of ambiguity aversion or preference for robustness. $\varepsilon\downarrow 0$ corresponds to the classical non-robust case since all measures other than $\mathbb{P}$ would give a very negative value, thus not optimal for the inner maximization problem. A larger $\varepsilon$ means the individual is more ambiguity averse, has less faith in the reference model and will consider larger drift distortion. $\varepsilon\rightarrow \infty$ corresponds to the worst-case approach, i.e. the individual has equal belief in all candidate measures and optimize again the worst-case scenario.

We now give the precise definition of the set $\mathscr{Q}$ of candidate measures. A probability measure $\mathbb{Q}\in\mathscr{Q}$ if
\begin{equation}\label{dQdP}
\frac{d\mathbb{Q}_t}{d\mathbb{P}_t}=\exp\left(-\frac{1}{2}\int_0^t \bftheta^2_s ds+\int_0^t \bftheta_s dB_s\right), \quad t\geq 0
\end{equation}
for some $\mathbb{F}$-progressively measurable process $\bftheta$ satisfying $\mathbb{E}[e^{\frac{1}{2}\int_0^t\bftheta_s^2ds}]<\infty$ for all $t\geq 0$, and $\mathbb{E}^{\mathbb{Q}}[\int_0^\infty e^{-\lambda s}\bftheta^2_s ds]<\infty$. Conversely, given any $\mathbb{F}$-progressively measurable process $\bftheta$ satisfying $\mathbb{E}[e^{\frac{1}{2}\int_0^t\bftheta_s^2ds}]<\infty$ for all $t\geq 0$, we can define a consistent family of measures $\mathbb{Q}_t\sim\mathbb{P}_t$ on $(\Omega, \mathcal{G}_t)$ by \eqref{dQdP}. By \cite[Lemma 4.2]{Stroock} (also see \cite[Proposition 1]{Huang92}), there exists a probability measure $\mathbb{Q}$ on $(\Omega,\mathcal{F})$ such that $\mathbb{Q}|_{\mathcal{G}_t}=\mathbb{Q}_t$ for all $t\geq 0$.\footnote{The existence of such a measure is not guaranteed if the filtration has been completed w.r.t. $\mathbb{P}$.} Throughout this paper, we will use boldface greeks $\bfpi, \bftheta$ to denote controls (as stochastic processes) and plain greeks $\pi, \theta$ to denote the values that the controls can take. Since $\tau_d$ is independent of $\mathbb{F}$, the distribution of $\tau_d$ is invariant under such change of measure. Under $\mathbb{Q}$, $S_t$ has drift $\mu+\sigma\bftheta_t$ and $W_t$ has dynamics:
\begin{align}
dW_t&=[rW_t+(\mu+\sigma\bftheta_t-r)\bfpi_t-c]dt+\sigma\bfpi_t dB^{\mathbb{Q}}_t \label{QW}
\end{align}
where $B^\mathbb{Q}$ is a $\mathbb{Q}$-Brownian motion independent of $\tau_d$.

Let $\mathbb{Q}\in\mathscr{Q}$. We have 
\begin{align*}
h^d(\mathbb{Q}| \mathbb{P})&=\mathbb{E}^{\mathbb{Q}}\left[-\frac{1}{2}\int_0^{\tau_d} \bftheta^2_s ds+\int_0^{\tau_d} \bftheta_s dB_s\right]\\
&=\mathbb{E}^{\mathbb{Q}}\left[-\frac{1}{2}\int_0^{\tau_d} \bftheta^2_s ds+\int_0^{\tau_d} \bftheta_s (dB^{\mathbb{Q}}_s+\bftheta_sds)\right]\\
&=\mathbb{E}^{\mathbb{Q}}\left[\frac{1}{2}\int_0^{\tau_d} \bftheta^2_s ds\right]=\mathbb{E}^{\mathbb{Q}}\left[\frac{1}{2}\int_0^\infty  e^{-\lambda s}\bftheta^2_s ds\right]<\infty.
\end{align*}
\begin{remark}
We can also compute the relative entropy process $h_t(\mathbb{Q}|\mathbb{P}):=h(\mathbb{Q}_t|\mathbb{P}_t)=\mathbb{E}^{\mathbb{Q}}\left[\frac{1}{2}\int_0^{t} \bftheta^2_s ds\right]$. Observe that
\begin{align*}
\mathbb{E}^{\mathbb{Q}}[h_{\tau_d}(\mathbb{Q}|\mathbb{P})]&=\mathbb{E}^{\mathbb{Q}}\left[\int_0^\infty \lambda e^{-\lambda t}h_t(\mathbb{Q}|\mathbb{P})dt\right]=\mathbb{E}^{\mathbb{Q}}\left[\int_0^\infty  \lambda e^{-\lambda t} \frac{1}{2} \int_0^{t}\bftheta^2_s ds dt\right]\\
&=\mathbb{E}^{\mathbb{Q}}\left[\frac{1}{2}\int_0^\infty \bftheta^2_s\int_s^{\infty}\lambda e^{-\lambda t} dt ds\right]=\mathbb{E}^{\mathbb{Q}}\left[\frac{1}{2}\int_0^\infty  e^{-\lambda s}\bftheta^2_s ds\right]=h^d(\mathbb{Q}| \mathbb{P}).
\end{align*}
So we can also think of $h^d$ as penalizing the expected relative entropy at death time.
\end{remark}

Substituting the expression for $h^d(\mathbb{Q}|\mathbb{P})$ into \eqref{optprob} and using the distribution of $\tau_d$, we rewrite the value function as:
\begin{defn}[Robust value function]
\begin{equation*}
\begin{aligned}
\psi(w;\varepsilon)
&=\inf_{\pi\in\mathscr{A}}\sup_{\mathbb{Q}\in\mathscr{Q}}\mathbb{E}^{\mathbb{Q}}_{w}\left[ \int_0^\infty e^{-\lambda s}\left( \lambda 1_{\{\tau_b<s\}}-\frac{1}{2\varepsilon}\bftheta_s^2\right)ds\right]
\end{aligned}
\end{equation*}
where $W$ has $\mathbb{Q}$-dynamics \eqref{QW}.
\end{defn}

Denote by $\psi_0$ the non-robust value function and by $\mathfrak{p}$ the robust value function when $\lambda=0$, i.e. when the individual never dies. $\psi_{0}$ has the explicit formula (see \cite{Young04}):
\begin{equation}\label{psiP}
\psi_{0}(w)=\begin{cases}
1, & w\leq b;\\
\left(\frac{c-rw}{c-rb}\right)^d, & b\leq w\leq c/r;\\
0, & w\geq c/r;
\end{cases}
\end{equation}
and the optimal investment strategy in feedback form is given by
\begin{equation*}
\pi_{0}(w)=\frac{\mu-r}{\sigma^2}\frac{c-rw}{(d-1)r}
\end{equation*}
for $w\in(b,w_s)$, where
\begin{equation}\label{Rd}
d=\frac{1}{2r}\left[(r+\lambda+R)+\sqrt{(r+\lambda+R)^2-4r\lambda}\right]>1, \quad R=\frac{1}{2}\left(\frac{\mu-r}{\sigma}\right)^2.
\end{equation}
Throughout this paper, $d$ and $R$ will be reserved for the constants defined above.
We will also provide an explicit formula for $\mathfrak{p}$ later. For now, we make the simple observation:
\begin{equation}\label{order_of_value_functions}
0\leq \psi_{0}\leq \psi\leq \mathfrak{p}\leq 1,
\end{equation}
where the second inequality holds because $\mathbb{P}\in\mathscr{Q}$ so that 
\[\psi_0=\inf_{\pi\in\mathscr{A}}\mathbb{P}_w(\tau_b<\tau_d)\leq \inf_{\pi\in\mathscr{A}}\sup_{\mathbb{Q}\in\mathscr{Q}} \left\{\mathbb{Q}_w(\tau_b<\tau_d)-\frac{1}{\varepsilon} h^d(\mathbb{Q}|\mathbb{P})\right\}=\psi,\]
the third inequality holds because ruin before death is no more likely than ruin before infinity, and the last inequality holds because we are optimizing a real probability minus a nonnegative penalty. This means we can treat the robust optimal value as a conservative ruin probability. The penalty term will only cause a small distortion on the ruin probability and will never drive it negative because only measures with small relative entropy are relevant, i.e. have the possibility of being worse than the reference measure. 

The definition of $\psi(w;\varepsilon)$ implies it is non-decreasing in $\varepsilon$, since the penalty gets smaller as $\varepsilon$ gets larger. We will suppress the argument $\varepsilon$ throughout the rest of this paper unless we need to emphasize the $\varepsilon$-dependence. The limit as $\varepsilon\downarrow 0$ gives us the non-robust value function $\psi_{0}$. The limit as $\varepsilon\rightarrow \infty$ gives us the worst-case value function:
\begin{equation*}
\psi_\infty(w):=\inf_{\pi\in\mathscr{A}}\sup_{\mathbb{Q}\in\mathscr{Q}} \mathbb{Q}_w(\tau_b<\tau_d).
\end{equation*}
For the worse-case problem, the optimal investment strategy is not to invest at all since the drift can be arbitrarily unfavorable (negative if one longs and positive if one shorts) without incurring any penalty. The individual can only hope to ``win" the game by dying quickly enough before consumption drags her wealth down to the ruin level. In this case, the agent's wealth solves the deterministic differential equation: 
\[dW_t=(rW_t-c)dt, \quad W_0=w.\]
Simple computation leads to $\tau_b=\frac{1}{r}\ln\frac{c-rb}{c-rw}$ and $\mathbb{Q}(\tau_b<\tau_d)=e^{-\lambda\tau_b}=\left(\frac{c-rw}{c-rb}\right)^{\frac{\lambda}{r}}$ for $w\in[b,w_s]$ and for all $\mathbb{Q}\in\mathscr{Q}$. So
\begin{equation}\label{psi_inf}
\psi_\infty(w)=\left(\frac{c-rw}{c-rb}\right)^{\frac{\lambda}{r}}, \quad w\in[b,w_s].
\end{equation}
Alternatively, we can obtain the above formula for $\psi_\infty$ by solving \eqref{BVP} with $\varepsilon$ set to infinity; a verification theorem has to be done then.

Back to the general case. $\psi(w)$ is non-increasing in $w$ since the individual is clearly better off with a larger initial wealth. When $w\leq b$, $\tau_b=0$ and $\psi(w)=1$ because the inner supremum can always be attained by the reference measure $\mathbb{P}$. Notice that by \eqref{order_of_value_functions}, we have continuity of $\psi$ at $w=b$ since $1\geq\lim_{w\rightarrow b}\psi(w)\geq\lim_{w\rightarrow b}\psi_{0}(w)=1$. Let $w_s:=c/r$. $w_s$ gives a ``safe" wealth level at which the individual can sustain her consumption by putting all her money in the bank and consuming the interest. This means $\psi(w)=0$ when $w\geq w_s$. Drift uncertainty is irrelevant here since the individual can always play safe by not investing in the risky asset. We also have continuity of $\psi$ at $w=w_s$ because $0\leq\lim_{w\rightarrow w_s}\psi(w)\leq\lim_{w\rightarrow w_s}\psi_{\infty}(w)=0$.

The associated HJB equation for $\psi$ in the interval $(b,w_s)$ is
\begin{equation}\label{Isaacs}
\lambda \psi(w)=\inf_\pi \sup_\theta\left\{-\frac{1}{2\varepsilon}\theta^2+\left(rw-c+(\mu+\sigma\theta-r)\pi\right)\psi'(w)+\frac{1}{2}\sigma^2\pi^2\psi''(w)\right\},
\end{equation}
with boundary conditions $\psi(b)=1$ and $\psi(w_s)=0$. Notice that the expression inside the braces is quadratic in $\theta$ with negative leading coefficient. By the first order condition, the optimal $\theta$ given $\pi$ equals $\sigma \varepsilon \pi \psi'$. Substituting $\theta=\sigma \varepsilon \pi \psi'$ back into \eqref{Isaacs}, we get
\begin{equation}\label{HJB}
\lambda \psi=\inf_{\pi}\left\{ \frac{1}{2}\sigma^2\left(\varepsilon(\psi')^2+\psi''\right)\pi^2+(\mu-r)\psi'\pi+\left(rw-c\right)\psi'\right\}.
\end{equation}
Suppose $\varepsilon(\psi')^2+\psi''>0$, we use first order condition again to find the candidate optimizer
\begin{equation}\label{pistar}
\pi^\ast = -\frac{\mu-r}{\sigma^2} \frac{\psi'}{\varepsilon (\psi')^2+\psi''}.
\end{equation}
It follows that 
\begin{equation}\label{thetastar}
\theta^\ast = - \frac{\mu-r}{\sigma} \frac{\varepsilon(\psi')^2}{\varepsilon (\psi')^2+\psi''}.
\end{equation}
Substituting \eqref{pistar} into \eqref{HJB}, we obtain the following Dirichlet boundary value problem:
\begin{subequations}
\label{BVP}
\begin{align}
&\lambda \psi=- \frac{R(\psi')^2}{\varepsilon(\psi')^2+\psi''}+\left(rw-c\right)\psi' \label{BVP:a}\\
&\psi(b)=1,\quad \psi(w_s)=0 \label{BVP:b}
\end{align}
\end{subequations}
where $R$ is the positive constant defined in \eqref{Rd}. When $\varepsilon=0$, we recover the non-robust value function $\psi_{0}$ whose formula is given in \eqref{psiP}. When $\varepsilon=\infty$, we get the worst-case value function $\psi_\infty$ whose formula is given in \eqref{psi_inf}.

\begin{remark}
The Isaacs condition does not hold for our robust problem without further restrictions on model parameters. Suppose $\psi''<0$ but $\varepsilon(\psi')^2+\psi''>0$, then maximizing over $\theta$ first and minimizing over $\pi$ second in \eqref{Isaacs} will lead to a finite Hamiltonian, but minimizing over $\pi$ first and maximizing over $\theta$ second will lead to an unbounded Hamiltonian. From another perspective, we expect the value function of each fixed-measure lifetime ruin problem to be convex, otherwise the Hamiltonian would explode. Maximizing over these convex functions will yield a convex function. On the other hand, our robust value function may be concave in certain region. When $r>\lambda$, the worst-case value function $\psi_\infty$ is concave. Since $\psi(w;\varepsilon)$ increases to $\psi_\infty(w)$ as $\varepsilon\rightarrow \infty$, $\psi(w;\varepsilon)$ cannot be convex everywhere for $\varepsilon$ sufficiently large. See Proposition \ref{psiconvex} for a more detailed discussion on how convexity depends on $\lambda$, $r$ and $\varepsilon$.
\end{remark}

Rigorous analysis of equation \eqref{HJB} will be done in Sections~\ref{sec:viscosity_soln} and ~\ref{sec:regularity}. Section~\ref{sec:psi0} provides an explicit solution to the Dirichlet problem \eqref{BVP} when $\lambda=0$. We end this section with our main result the proof of which is given at the end of Section \ref{sec:verification}. 
\begin{thm}\label{main_thm}
The robust value function $\psi$ satisfies $\psi(w)=1$ for $w\leq b$, $\psi(w)=0$ for $w\geq w_s$. For $w\in(b,w_s)$, $\psi(w)$ is the unique $C^1[b,w_s]\cap C^2[b,w_s)$ solution to \eqref{Isaacs} or \eqref{HJB} satisfying the boundary conditions $\psi(b)=1$ and $\psi(w_s)=0$. The optimal investment policy is
\[\bfpi^\ast_t=-\frac{\mu-r}{\sigma^2}\frac{\psi'(W_t)}{\varepsilon(\psi'(W_t))^2+\psi''(W_t)}1_{(b,w_s)}(W_t),\]
and the optimal drift distortion is $\sigma\bftheta^\ast$ where
\[\bftheta^\ast_t=-\frac{\mu-r}{\sigma}\frac{\varepsilon(\psi'(W_t))^2}{\varepsilon(\psi'(W_t))^2+\psi''(W_t)}1_{(b,w_s)}(W_t).\]
\end{thm}

\section{Explicit solution for the $\lambda=0$ case}\label{sec:psi0}

Setting $\lambda=0$ in \eqref{BVP}, we get
\begin{equation}\label{BVP0}
\begin{aligned}
&0=- \frac{R(\psi')^2}{\epsilon(\psi')^2+\psi''}+\left(rw-c\right)\psi' \\
&\psi(b)=1,\quad \psi(w_s)=0.
\end{aligned}
\end{equation}
Using the exponential transformation $\phi=e^{\varepsilon \psi}$, also called Cole-Hopf transformation in PDE theory, the nonlinearity in the denominator is removed and \eqref{BVP0} becomes
\begin{equation*}
\begin{aligned}
&0=-R \frac{(\phi')^2}{\phi''}+\left(rw-c\right)\phi'\\
&\phi(b)=e^{\varepsilon},\quad \phi(w_s)=1.
\end{aligned}
\end{equation*}
Suppose $\phi'\neq 0$ and let $u=\phi'$. The second order ordinary differential equation (ODE) is further reduced to
\begin{equation*}
u'=\frac{R}{rw-c}u,
\end{equation*}
the general solution of which is given by
\[u(w)=Ae^{R\int_b^w \frac{1}{rz-c}dz}=A\left(\frac{c-rw}{c-rb}\right)^{\frac{R}{r}}, \quad A\in\mathbb{R}.\]
It follows that
\begin{equation*}
\phi(w)=e^{\varepsilon}+A\int_b^w \left(\frac{c-rz}{c-rb}\right)^{\frac{R}{r}} dz=e^{\varepsilon}-A\frac{c-rb}{R+r}\left[\left(\frac{c-rw}{c-rb}\right)^{\frac{R}{r}+1}-1\right].
\end{equation*}
Using the boundary condition at the safe level, we can determine the constant $A$ and obtain
\begin{equation*}
\phi(w)=1+(e^\varepsilon-1)\left(\frac{c-rw}{c-rb}\right)^{\frac{R}{r}+1}.
\end{equation*}
So the solution to the Dirichlet problem \eqref{BVP0} is
\begin{equation}\label{psi0}
\psi(w)=\frac{1}{\varepsilon}\ln\left[1+(e^\varepsilon-1)\left(\frac{c-rw}{c-rb}\right)^{\frac{R}{r}+1}\right].
\end{equation}
The feedback forms \eqref{pistar}, \eqref{thetastar} become
\begin{equation*}
\varpi=\frac{2(c-rw)}{\mu-r},
\end{equation*}
\begin{equation*}
\vartheta=-\frac{2\sigma(R+r)}{\mu-r}\frac{(e^\varepsilon-1)\left(\frac{c-rw}{c-rb}\right)^{\frac{R}{r}+1}}{1+(e^\varepsilon-1)\left(\frac{c-rw}{c-rb}\right)^{\frac{R}{r}+1}}.
\end{equation*}

The solution given by \eqref{psi0} is a $C^1[b,w_s]\cap C^2[b,w_s)$ function. $\varpi$ and $\vartheta$ are bounded, Lipschitz continuous functions of the state variable on $[b,w_s]$. So a verification theorem can be easily done, showing the function given by \eqref{psi0} is indeed the robust value function $\mathfrak{p}$ on the interval $[b,w_s]$, and $\varpi, \vartheta$ are the optimal feedback controls. We summarize the results in the following theorem.

\begin{thm}
When $\lambda=0$, the robust value function is given by
\[\mathfrak{p}(w)=\frac{1}{\varepsilon}\ln\left[1+(e^\varepsilon-1)\left(\frac{c-rw}{c-rb}\right)^{\frac{R}{r}+1}\right]\]
for $b\leq w\leq w_s$, $\mathfrak{p}(w)=0$ for $w\leq b$ and $\mathfrak{p}(w)=1$ for $w\geq w_s$. The optimal investment policy is
\[\bfvarpi_t=\frac{2(c-rW_t)}{\mu-r}1_{(b,w_s)}(W_t),\]
and the optimal drift distortion is $\sigma\bfvartheta$ where
\[\bfvartheta_t=-\frac{2\sigma(R+r)}{\mu-r}\frac{(e^\varepsilon-1)\left(\frac{c-rW_t}{c-rb}\right)^{\frac{R}{r}+1}}{1+(e^\varepsilon-1)\left(\frac{c-rW_t}{c-rb}\right)^{\frac{R}{r}+1}}1_{(b,w_s)}(W_t).\]
\end{thm}

One observation is the loss of convexity of the value function compared with the non-robust case. This is caused by the nonlinear term $\varepsilon (\psi')^2$. When $\varepsilon$ is zero, $\psi''$ must be non-negative (in fact, strictly positive if $\psi'\neq 0$) for the Hamiltonian in \eqref{HJB} to be finite. When $\varepsilon$ is nonzero, $\psi''$ is allowed to take negative values as long as $\varepsilon (\psi')^2+\psi''$ is non-negative. The larger the $\varepsilon$, the more concave the value function could potentially be. Another interesting feature is that when hazard rate is zero, the pre-ruin optimal investment policy is independent of both the ambiguity aversion parameter $\varepsilon$ and the ruin level $b$. Also, we see that for $w\in(b,w_s)$, 
$\lim_{\varepsilon\rightarrow \infty}\vartheta(w)=-\frac{2\sigma(R+r)}{\mu-r}.$ In terms of the optimally distorted Sharpe ratio, we have
\[\lim_{\varepsilon\rightarrow \infty}\left(\frac{\mu-r}{\sigma}+\vartheta(w)\right)=-\frac{2\sigma r}{\mu-r}. \]
Figure \ref{lambda0} shows plots for the robust ruin probability $\mathfrak{p}$ and the optimally distorted Sharpe ratio $\frac{\mu-r}{\sigma}+\vartheta$ with parameters $c=1$, $b=1$, $r=0.02$, $\mu=0.1$, $\sigma=0.15$ and $\varepsilon=0,1,5,10,50$.
We leave out the plot for $\varpi$ since it is a simple downward sloping linear function, and is independent of $\varepsilon$. It is worth mentioning that $\varpi\geq \pi_{0}$, i.e. the individual adopts a more aggressive investment strategy when life is perpetual.

\begin{figure}[h]
\centering
\includegraphics[height=6cm]{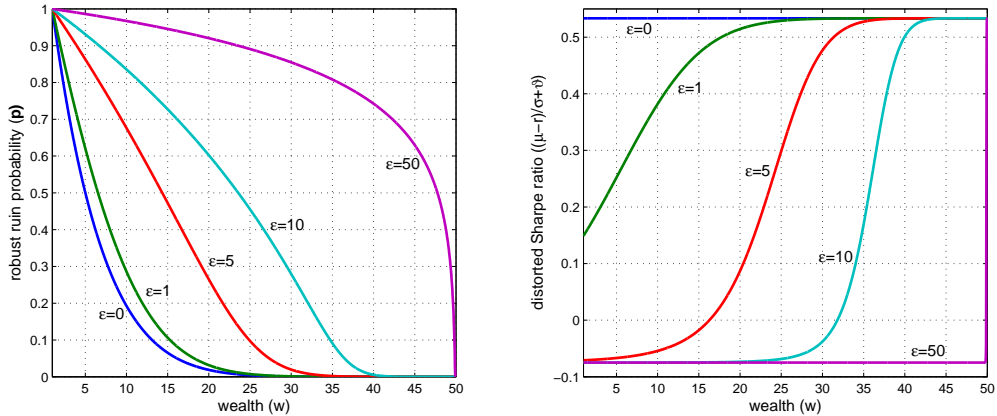}
\caption{Robust ruin probability and optimally distorted Sharpe ratio when $\lambda=0$.}
\label{lambda0}
\end{figure}

Before we move on to the general case, let us make one more remark regarding differentiability at the safe level.
\begin{remark}\label{rmk_zero_derivative_at_ws}
From the explicit formula for $\mathfrak{p}$, we see that $\mathfrak{p}$ has zero derivative at the safe level. Since $\mathfrak{p}$ bounds any general $\psi$ from above, this property is also shared by $\psi$. Indeed,
\[0\geq\lim_{w\rightarrow w_s-}\frac{\psi(w)}{w-w_s}\geq\lim_{w\rightarrow w_s-}\frac{\mathfrak{p}(w)}{w-w_s}=\mathfrak{p}'(w_s)=0.\]
\end{remark}

\section{Viscosity Solution and Perron's Method}\label{sec:viscosity_soln}

Our goal in this section is to show the nonlinear degenerate elliptic Dirichlet problem
\begin{subequations}
\label{P1}
\begin{align}
&F(w,u,u',u'')=0,\label{P1:a}\\
&u(b)=1, \ u(w_s)=0, \label{P1:b}
\end{align}
\end{subequations}
where
\begin{align*}
F(w,u,u',u''):=\lambda u-\inf_{\pi}\left\{\frac{1}{2}\sigma^2\left(\varepsilon (u')^2+u''\right)\pi^2+(\mu-r)u'\pi+\left(rw-c\right)u'\right\}
\end{align*}
has a unique viscosity solution satisfying certain properties. Notice that $F$ can be written as the supremum of a family of continuous functions, hence is lower semi-continuous (l.s.c.). 

We first prove a comparison principle for \eqref{P1} which implies uniqueness. The proof is a slight modification of the classical comparison argument to take care of the unboundedness of the control space. It turns out, luckily, that the nonlinear term $\varepsilon(u')^2$ does not add any difficulty.

\begin{prop}\label{comparisonP1}
Let $u, v$ be an upper semi-continuous (u.s.c.) viscosity subsolution and a l.s.c. viscosity supersolution of $F=0$, respectively. Suppose $u, v$ are bounded, and either $v\geq 0$ or $u>0$ in $(b,w_s)$. If $u\leq v$ on $\partial (b,w_s)$, then $u\leq v$ on $[b,w_s]$.
\end{prop}
\begin{proof}
Suppose that, on the contrary, $\delta:=\sup_{x\in(b,w_s)}(u-v)(x)>0$. $\delta<\infty$ since $u$ and $v$ are assumed to be bounded. By the upper semi-continuity of $u-v$, there exists $x^\ast\in(b,w_s)$ such that $u(x^\ast)-v(x^\ast)=\delta$. For every $\alpha>0$, define 
\[\Psi_\alpha(x,y):=u(x)-v(y)-\frac{\alpha}{2}|x-y|^2.\]
It is clear that $\sup_{x,y\in(b,w_s)}\Psi_{\alpha}(x,y)\geq \delta$ since we can always choose $x=y$. By the upper semi-continuity of $u(x)-v(y)$, there exists $\hat{x}_\alpha, \hat{y}_\alpha$ such that $\sup_{x,y\in(b,w_s)}\Psi_\alpha(x,y)=\Psi_\alpha(\hat{x}_\alpha,\hat{y}_\alpha)$. We have
\[u(x^\ast)-v(x^\ast)\leq u(\hat{x}_\alpha)-v(\hat{y}_\alpha)-\frac{\alpha}{2}|\hat{x}_\alpha-\hat{y}_{\alpha}|^2.\]
This implies
\begin{equation}\label{comparison_eq1}
\frac{\alpha}{2}|\hat{x}_\alpha-\hat{y}_{\alpha}|^2\leq u(\hat{x}_\alpha)-v(\hat{y}_\alpha)-(u(x^\ast)-v(x^\ast)).
\end{equation}
Since $[b,w_s]$ is compact, we can find a sequence $\alpha_n\rightarrow\infty$ such that $(\hat{x}_n, \hat{y}_n):=(\hat{x}_{\alpha_n}, \hat{y}_{\alpha_n})$ converges to $(\hat{x}, \hat{y})$ as $n\rightarrow \infty$.
Replacing $\alpha$ by $\alpha_n$ and letting $n\rightarrow\infty$ in \eqref{comparison_eq1}, we obtain
\begin{equation}\label{comparison_eq2}
\begin{aligned}
\limsup_n\frac{\alpha_n}{2}|\hat{x}_n-\hat{y}_n|^2&\leq \limsup_n (u(\hat{x}_n)-v(\hat{y}_n))-(u(x^\ast)-v(x^\ast))\\
&\leq u(\hat{x})-v(\hat{y})-(u(x^\ast)-v(x^\ast)),
\end{aligned}
\end{equation}
where the second inequality is due to the upper semi-continuity of $u(x)-v(y)$. Since the right hand side of \eqref{comparison_eq2} is finite and $\alpha_n\rightarrow \infty$, we must have $\hat{x}=\hat{y}$, and \eqref{comparison_eq2} yields
\begin{equation*}
\begin{aligned}
0\leq\limsup_n\frac{\alpha_n}{2}|\hat{x}_n-\hat{y}_n|^2
&\leq u(\hat{x})-v(\hat{x})-(u(x^\ast)-v(x^\ast))\leq 0,
\end{aligned}
\end{equation*}
which implies $u(\hat{x})-v(\hat{x})=u(x^\ast)-v(x^\ast)=\delta$, $\alpha_n|\hat{x}_n-\hat{y}_n|^2\rightarrow 0$ and
\begin{equation}\label{comparison_eq3}
\delta\leq \sup_{x,y\in(b,w_s)}\Psi_{\alpha_n}(x,y)=u(\hat{x}_n)-v(\hat{y}_n)-\frac{\alpha_n}{2}|\hat{x}_n-\hat{y}_n|^2\rightarrow u(x^\ast)-v(x^\ast)=\delta
\end{equation}
as $n\rightarrow \infty$. Now, since $u\leq v$ on $\partial (b,w_s)$, we must have $\hat{x}\in (b,w_s)$. So $\hat{x}_n, \hat{y}_n\in(b,w_s)$ for sufficiently large $n$. By Crandall-Ishii's lemma, we can find sequences $A_n, B_n$ satisfying $-3\alpha_n\leq A_n\leq B_n\leq 3\alpha_n$ and
\begin{equation*}
(\alpha_n(\hat{x}_n-\hat{y}_n), A_n)\in \bar{J}^{2,+}_{(b,w_s)}u(\hat{x}_n), \quad (\alpha_n(\hat{x}_n-\hat{y}_n), B_n)\in \bar{J}^{2,-}_{(b,w_s)}v(\hat{y}_n),
\end{equation*}
where $\bar{J}^{2,+}_{(b,w_s)}u(\hat{x}_n), \bar{J}^{2,-}_{(b,w_s)}v(\hat{y}_n)$ are the closure of the second order superjet and subjet, respectively. Since $u$ is a viscosity subsolution of $F=0$ and $F$ is l.s.c., we have by \cite[Proposition 6.11.i]{Touzi} that
\begin{equation}\label{comparison_eq4}
F(\hat{x}_n,u(\hat{x}_n),\alpha_n(\hat{x}_n-\hat{y}_n),A_n)\leq 0.
\end{equation}  
The finiteness of $F(\hat{x}_n,u(\hat{x}_n),\alpha_n(\hat{x}_n-\hat{y}_n),A_n)$ implies either $\varepsilon \alpha^2_n(\hat{x}_n-\hat{y}_n)^2+A_n>0$ or $\varepsilon \alpha^2_n(\hat{x}_n-\hat{y}_n)^2+A_n=\alpha_n(\hat{x}_n-\hat{y}_n)=0$. We consider each case separately.

Case 1. $\varepsilon \alpha^2_n(\hat{x}_n-\hat{y}_n)^2+A_n>0$. In this case, we also have $\varepsilon \alpha^2_n(\hat{x}_n-\hat{y}_n)^2+B_n>0$. Since $F(w,u,u',u'')$ is continuous in the region $\varepsilon (u')^2+u''>0$, the supersolution property of $v$ implies
\begin{equation*}
F(\hat{y}_n,v(\hat{y}_n),\alpha_n(\hat{x}_n-\hat{y}_n),B_n)\geq 0.
\end{equation*}  
(See \cite[Proposition 6.11.ii]{Touzi}.) So we have
\begin{equation}\label{comparison_eq4b}
F(\hat{x}_n,u(\hat{x}_n),\alpha_n(\hat{x}_n-\hat{y}_n),A_n)\leq 0\leq F(\hat{y}_n,v(\hat{y}_n),\alpha_n(\hat{x}_n-\hat{y}_n),B_n)<\infty.
\end{equation}  
Using the expression of $F$, we obtain from \eqref{comparison_eq3} and \eqref{comparison_eq4b} that
\begin{align*}
\lambda \delta\leq \lambda (u(\hat{x}_n)-v(\hat{y}_n))&=F(\hat{x}_n,u(\hat{x}_n),\alpha_n(\hat{x}_n-\hat{y}_n),A_n)-F(\hat{x}_n,v(\hat{y}_n),\alpha_n(\hat{x}_n-\hat{y}_n),A_n)\\
&\leq F(\hat{y}_n,v(\hat{y}_n),\alpha_n(\hat{x}_n-\hat{y}_n),B_n)-F(\hat{x}_n,v(\hat{y}_n),\alpha_n(\hat{x}_n-\hat{y}_n),A_n)\\
&=\frac{R\alpha_n^2(\hat{x}_n-\hat{y}_n)^2(A_n-B_n)}{[\varepsilon\alpha_n^2(\hat{x}_n-\hat{y}_n)^2+B_n][\varepsilon\alpha_n^2(\hat{x}_n-\hat{y}_n)^2+A_n]}+r\alpha_n(\hat{x}_n-\hat{y}_n)^2\\
&\leq r\alpha_n(\hat{x}_n-\hat{y}_n)^2.
\end{align*}
Letting $n\rightarrow \infty$, we arrive at the contradiction $\lambda\delta\leq 0$.

Case 2. $\varepsilon \alpha^2_n(\hat{x}_n-\hat{y}_n)^2+A_n=\alpha_n(\hat{x}_n-\hat{y}_n)=0$. In this case, 
Equation \eqref{comparison_eq4} reads
\begin{equation*}
\lambda u(\hat{x}_n)=F(\hat{x}_n,u(\hat{x}_n),\alpha_n(\hat{x}_n-\hat{y}_n),A_n)\leq 0.
\end{equation*}
If $u$ is strictly positive, this cannot happen. So assume we are in the case where $v$ is non-negative. But this implies $u(\hat{x}_n)-v(\hat{x}_n)\leq 0$, contradicting $u(\hat{x}_n)-v(\hat{y}_n)\geq \sup_{x,y\in(b,w_s)}\Psi_{\alpha_n}(x,y)\geq\delta$.
\end{proof}

\begin{cor}
There is at most one viscosity solution to the Dirichlet problem \eqref{P1} that is bounded, non-negative, and continuous at the boundary.
\end{cor}

\begin{lemma}\label{max_stability_of_subsoln}
Let $\mathcal{U}$ be a non-empty family of u.s.c. viscosity subsolutions of $F=0$. Define
\[\mathbf{u}(w)=\sup_{u\in\mathcal{U}} u(w).\]
Let $\mathbf{u}^\ast$ be the u.s.c. envelope of $\mathbf{u}$ and assume $\mathbf{u}^\ast(w)<\infty$ for $w\in(b,w_s)$. Then $\mathbf{u}^\ast$ is a viscosity subsolution of $F=0$.
\end{lemma}
\begin{proof}
By Lemma 4.2 of \cite{UsersGuide}. Note that although their function $F$ is $\mathbb{R}$-valued, the proof works exactly the same way when $F$ is allowed to take $\infty$ as a value as long as it is l.s.c. which is satisfied in our case.
\end{proof}

Next, we use Perron's Method to show problem \eqref{P1} has a viscosity solution. We mimic the proof of \cite[Theorem 4.1]{UsersGuide}. To begin with, we need to find an (u.s.c.) viscosity subsolution whose l.s.c. envelope satisfies the boundary conditions \eqref{P1:b}, and a (l.s.c.) viscosity supersolution whose u.s.c. envelope satisfies the boundary conditions \eqref{P1:b}. Obviously, we should aim at those functions that bound the robust value function from below and above, and we have two natural candidates: $\psi_{0}$ and $\mathfrak{p}$.\footnote{We can also use $\psi_{\infty}$ as the upper bound.} Indeed,
\begin{align*}
F(w,\psi_{0}, \psi'_{0}, \psi''_{0})&=\lambda \psi_{0}-\inf_{\pi}\left\{\frac{1}{2}\sigma^2\left(\varepsilon (\psi'_{0})^2+\psi''_{0}\right)\pi^2+(\mu-r)\psi'_{0}\pi+\left(rw-c\right)\psi'_{0}\right\}\\
&=\lambda \psi_{0}+ \frac{R(\psi'_{0})^2}{\varepsilon(\psi'_{0})^2+\psi''_{0}}-\left(rw-c\right)\psi'_{0}\\
&\leq \lambda \psi_{0}+\frac{R(\psi'_{0})^2}{\psi''_{0}}-\left(rw-c\right)\psi'_{0}=0,
\end{align*}
where in the second equality we used $\psi''_{0}>0$ in $(b,w_s)$, and
\begin{align*}
F(w,\mathfrak{p}, \mathfrak{p}', \mathfrak{p}'')&=\lambda \mathfrak{p}-\inf_{\pi}\left\{\frac{1}{2}\sigma^2\left(\varepsilon (\mathfrak{p}')^2+\mathfrak{p}''\right)\pi^2+(\mu-r)\mathfrak{p}'\pi+\left(rw-c\right)\mathfrak{p}'\right\}\\
&=\lambda \mathfrak{p}\geq 0.
\end{align*}
\begin{remark}
If these natural candidates were not available, we could start with the constant subsolution $u\equiv 0$ (resp. supersolution $v\equiv 1$), and modify it near the ruin level (resp. safe level) by a construction similar to that on page 25 of \cite{UsersGuide} so that the boundary conditions are satisfied.
\end{remark}

%%%%%%%%%%%%%%%%

\begin{prop}[Perron's method]
There exists a continuous viscosity solution to the Dirichlet problem \eqref{P1} that takes values in $[0,1]$. More precisely, it is bounded from below by $\psi_{0}$ and from above by $\mathfrak{p}$. 
\end{prop}
\begin{proof}
Let $\lbar{u}=\psi_0$ and $v=\mathfrak{p}$. Both are $[0,1]$-valued continuous functions. Define
\begin{equation}\label{Perron_def_bfu}
\mathbf{u}(w):=\sup\{u(w): \lbar{u}\leq u\leq v \text{ and $u$ is an u.s.c. subsolution of } F=0\}.
\end{equation}
For any function $u$, denote by $u^\ast$ and $u_\ast$ its u.s.c. envelope and l.s.c. envelope, respectively.
We have $\lbar{u}=\lbar{u}_\ast\leq \mathbf{u}_\ast\leq\mathbf{u}\leq \mathbf{u}^\ast \leq v^\ast=v$. Since $\lbar{u}$ and $v$ agree on the boundary, we know $\mathbf{u}$ is continuous at the boundary and satisfies the boundary condition \eqref{P1:b}. Since $\mathbf{u}^\ast\leq v<\infty$, Lemma \ref{max_stability_of_subsoln} implies $\mathbf{u}^\ast$ is a viscosity subsolution of $F=0$. If we can show $\mathbf{u}_\ast$ is a viscosity supersolution of $F=0$, we can then apply comparison principle to get $\mathbf{u}^\ast\leq \mathbf{u}_\ast$, and conclude that $\mathbf{u}$ is a continuous viscosity solution to the Dirichlet problem \eqref{P1}. The rest is devoted to the proof of the supersolution property of $\mathbf{u}_\ast$.

Suppose $\mathbf{u}_\ast$ is not a viscosity supersolution of $F=0$. Then there exists $w_0\in(b,w_s)$ and $\varphi\in C^2(b,w_s)$ such that $\mathbf{u}_\ast-\varphi$ has a strict minimum zero at $w_0$ and $F(w_0, \varphi(w_0),\varphi'(w_0), \varphi''(w_0))<0$. Here $F<\infty$ implies either $\varepsilon(\varphi'(w_0))^2+\varphi''(w_0)>0$ or $\varphi''(w_0)=\varphi'(w_0)=0$. In the latter case, we get $\mathbf{u}_\ast(w_0)=\varphi(w_0)<0$ which cannot happen because $\mathbf{u}_\ast\geq \lbar{u}\geq 0$. So we are in the former case. By continuity of $F$ in the region $\varepsilon(u')^2+u''>0$, there exists $\delta,\gamma>0$ such that $F(w,\varphi(w)+\gamma,\varphi'(w),\varphi''(w))<0$ for all $w\in B_{\delta}(w_0)\subset \ubar{B}_{\delta}(w_0)\subset (b,w_s)$. Let $\varphi_\gamma(w):=\varphi(w)+\gamma$. Then $\varphi_\gamma$ is a classical subsolution of $F=0$ in $B_{\delta}(w_0)$. Since $\mathbf{u}_\ast>\varphi$ in $(b,w_s)\backslash \{w_0\}$, we can choose $\gamma$ small so that $\mathbf{u}_\ast>\varphi+\gamma=\varphi_\gamma$ on $\partial B_{\delta}(w_0)$. Define
\[U:=\begin{cases} \mathbf{u}^\ast\vee \varphi_\gamma &\text{in } B_{\delta}(w_0),\\
\mathbf{u}^\ast & \text{otherwise.}
\end{cases}\]
Since $\mathbf{u}^\ast<\infty$ and $\varphi_\gamma\leq\mathbf{u}_\ast+\gamma<\infty$ in $B_{\delta}(w_0)$, by Lemma \ref{max_stability_of_subsoln}, $U^\ast$ is a viscosity subsolution of $F=0$. Since $U^\ast=\mathbf{u}^\ast\leq v$ on $\partial(b,w_s)$, comparison principle (Proposition \ref{comparisonP1}) implies $U^\ast\leq v$ on $[b,w_s]$. So $U^\ast$ belongs to the set on the right hand side of \eqref{Perron_def_bfu}, and thus $\mathbf{u}^\ast\leq U\leq U^\ast\leq \mathbf{u} \leq\mathbf{u}^\ast$, where the second last inequality is due to the maximality of $\mathbf{u}$. Therefore, we obtain $U=\mathbf{u}^\ast$. 

On the other hand, by the definition of the semi-continuous envelope, there exists a sequence $(w_n)\subset B_{\delta}(w_0)$ such that $w_n\rightarrow w_0$ and $\mathbf{u}^\ast(w_n)\rightarrow \mathbf{u}_\ast(w_0)$. It follows that
$\varphi_\gamma(w_n)-\mathbf{u}^\ast(w_n)=\varphi(w_n)+\gamma-\mathbf{u}^\ast(w_n)\rightarrow \gamma>0$. So for $n$ sufficiently large, $U(w_n)=\varphi_\gamma(w_n)>\mathbf{u}^\ast(w_n)+\gamma/2$. We get a contradiction. This completes the proof that $\mathbf{u}_\ast$ is a viscosity supersolution of $F=0$. 
\end{proof}

Up to this point, we have established the existence and uniqueness of a continuous viscosity solution to the Dirichlet problem \eqref{P1}. Denote this solution by $\hat{u}$. We have $\psi_{0}\leq \hat{u}\leq \mathfrak{p}$. The next goal is to upgrade regularity.

\section{Regularity}\label{sec:regularity}

One difficulty of directly proving regularity for problem \eqref{P1} is the lack of convexity of $\hat{u}$ caused by the nonlinear term $\varepsilon(u')^2$. Motivated by how we solved the $\lambda=0$ case, we use the Cole-Hopf transformation $v=e^{\varepsilon u}$ to obtain an equivalent convex problem:
\begin{subequations}
\label{P2}
\begin{align}
&G(w,v,v',v'')=0,\label{P2:a}\\
& v(b)=e^{\varepsilon}, \ v(w_s)=1.\label{P2:b}
\end{align}
\end{subequations}
where
\[G(w,v,v',v''):=\lambda v\ln v-\inf_{\pi}\left\{\frac{1}{2}\sigma^2v''\pi^2+(\mu-r)v'\pi+\left(rw-c\right)v'\right\}.\]
The solution to the transformed problem is expected to be convex, otherwise $G$ would explode. Although \eqref{P2:a} is only understood in viscosity sense for now, one can expect, intuitively, if at every interior point, every test function above the viscosity solution (for the subsolution property) is convex in a neighborhood of that point, then the viscosity solution should be convex as well.

Since we already have a continuous viscosity solution $\hat{u}$ of problem \eqref{P1}, it can be easily verified that $\hat{v}:=e^{\varepsilon \hat{u}}$ is a continuous viscosity solution of problem \eqref{P2} satisfying $e^{\varepsilon \psi_{0}}\leq\hat{v}\leq e^{\varepsilon \mathfrak{p}}$. Moreover, the comparison principle for \eqref{P1} immediately yields a comparison principle for \eqref{P2}. We summarize these results in the following two lemmas. 

\begin{lemma}\label{comparisonP2}
Let $u, v$ be strictly positive u.s.c. viscosity subsolution and l.s.c. viscosity supersolution of \eqref{P2:a}, respectively. Suppose $u, v$ are bounded and bounded away from zero and either $u>1$ or $v\geq 1$ in $(b,w_s)$. If $u\leq v$ on $\partial (b,w_s)$, then $u\leq v$ on $[b,w_s]$.
\end{lemma}
\begin{proof}
It is easy to check $\frac{1}{\varepsilon} \ln u$ (resp. $\frac{1}{\varepsilon} \ln v$) is an u.s.c. subsolution (resp. a l.s.c. supersolution) of \eqref{P1} satisfying all assumptions of Proposition \ref{comparisonP1}.
\end{proof}

\begin{lemma}\label{uniqueness_existence_P2}
$\hat{v}:=e^{\varepsilon \hat{u}}$ is the unique (continuous) viscosity solution to the Dirichlet problem \eqref{P2} among all viscosity solutions that are bounded, continuous at the boundary and satisfy $v\geq 1$ in $(b,w_s)$. Moreover, $e^{\varepsilon \psi_0}\leq\hat{v}\leq e^{\varepsilon \mathfrak{p}}$.
\end{lemma}

We now establish the convexity and monotonicity of $\hat{v}$.

\begin{lemma}\label{vhat_convex_monotone}
$\hat{v}$ is strictly convex and strictly decreasing on $[b,w_s]$.
\end{lemma}
\begin{proof}
First, let us show (non-strict) interior convexity. Suppose $\hat{v}$ is not convex in $(b,w_s)$. Then by \cite[Lemma 1]{Lion97}, there exists $w_0\in(b,w_s)$ and $(p, A)\in J^{2,+}\hat{v}(w_0)$ with $A<0$. We therefore have $G(w_0, \hat{v}(w_0), p, A)=\infty$. But by the semi-jets formulation of viscosity solution (see e.g. \cite[Proposition 6.11.i]{Touzi}), we have $G(w_0, \hat{v}(w_0), p, A)\leq 0$. We get a contradiction. Since $\hat{v}$ is continuous, interior convexity can be extended to the boundary. 

The convexity of $\hat{v}$ implies its left and right derivatives $D^{\pm}\hat{v}$ exists (in $\mathbb{R}$ for interior points and in $\mathbb{R}\cup\{\pm\infty\}$ for boundary points) and are non-decreasing.\footnote{Here and in the sequel, at the left (resp. right) boundary point, $D^{\pm}\hat{v}$ only refers to the right (resp. left) derivative.} Since we have showed $0\leq\hat{u}\leq \mathfrak{p}$ and we know $\mathfrak{p}'_0(w_s)=0$, an argument exactly the same as Remark \ref{rmk_zero_derivative_at_ws} yields $D^-\hat{u}(w_s)=0$. It follows that $D^-\hat{v}(w_s)=\varepsilon D^-\hat{u}(w_s)\hat{v}(w_s)=0$. So $D^{\pm}\hat{v}\leq \hat{v}'(w_s)=0$. Suppose $D^{+}\hat{v}(w_0)=0$ for some $w_0\in[b,w_s)$ (same if $D^{-}\hat{v}(w_0)=0$). Then by monotonicity of $D^{+}\hat{v}$, $D^{+}\hat{v}(w)=0 \ \forall w\in[w_0, w_s)$. By convexity, 
\[0=D^+\hat{v}(w)\leq\frac{\hat{v}(w)-\hat{v}(w_s)}{w-w_s}= \frac{\hat{v}(w)-1}{w-w_s}\leq 0 \quad\forall w\in[w_0,w_s).\]
We deduce $\hat{v}\equiv 1$ on $[w_0, w_s]$, contradicting the property that $\hat{v}\geq e^{\varepsilon \psi_0}>1$ in $(b,w_s)$ (see Lemma \ref{uniqueness_existence_P2}). Therefore, we must have $D^{\pm}\hat{v}<0$ in $[b,w_s)$ which implies $\hat{v}$ is strictly decreasing.

Finally, if $\hat{v}$ is convex but not strictly convex, then it is linear in some open interval $(x,y)\subset (b,w_s)$. Since $\hat{v}$ is strictly decreasing, the line has non-zero slope, say $p$. But this cannot happen because $G(w,\hat{v}(w),p,0)$ is unbounded. 
\end{proof}

Being a convex function, $\hat{v}$ has many nice regularity properties. It is differentiable almost everywhere (a.e.), and even twice differentiable a.e. by Alexandroff's classical result \cite{Alexandroff}. To show $C^2$-regularity, we first show $C^1$-regularity using properties of viscosity solution and then upgrade to $C^2$ by analyzing a Poisson equation with the non-homogeneous term expressed in terms of $\hat{v}$ and its first derivative.

\begin{lemma}\label{fv}
$(rw-c)D^{\pm}\hat{v}-\lambda \hat{v}\ln\hat{v}$ is non-negative for all $w\in(b,w_s)$, and strictly positive if $w$ is a point of twice differentiability of $\hat{v}$. 
\end{lemma}
\begin{proof}
By Lemma 2 in \cite{Lion97}, $G(w,\hat{v}(w),\hat{v}'(w),\hat{v}''(w))\leq 0$ at every point $w\in(b,w_s)$ of twice differentiability. Here we note that their lemma is stated for continuous $G$, but it can be easily modify to accommodate our l.s.c. $G$. Let $w\in(b,w_s)$ be a point where $\hat{v}$ is twice differentiable. Since $\hat{v}'(w)<0$, $G(w,\hat{v}(w),\hat{v}'(w),\hat{v}''(w))\leq 0$ implies $\hat{v}''(w)>0$, and
\begin{equation*}
\lambda \hat{v}(w)\ln\hat{v}(w)+R\frac{(\hat{v}'(w))^2}{\hat{v}''(w)}-(rw-c)\hat{v}'(w)\leq 0.
\end{equation*}
We get
\[(rw-c)\hat{v}'(w)-\lambda \hat{v}(w)\ln\hat{v}(w)\geq R\frac{(\hat{v}'(w))^2}{\hat{v}''(w)}>0.\]
For arbitrary $w\in(b,w_s)$, since $\hat{v}$ is twice differentiable a.e., we can find a sequence of twice differentiability points $(w_n)\subset (b,w_s)$ which converges to $w$ from the right. Using the monotonicity of $D^{\pm}\hat{v}$, we have
\[(rw_n-c)D^{\pm}\hat{v}(w)\geq (rw_n-c)\hat{v}'(w_n)>\lambda \hat{v}(w_n)\ln\hat{v}(w_n)\]
We are done by letting $n\rightarrow\infty$ and using the continuity of $\hat{v}$.
\end{proof}

\begin{lemma}\label{vhatC1}
$\hat{v}\in C^1[b,w_s]$.
\end{lemma}
\begin{proof}
We first show interior $C^1$-regularity. It suffices to show $\hat{v}$ is differentiable since a convex differentiable function is continuously differentiable. Suppose on the contrary, $D^-\hat{v}(w_0)\neq D^+\hat{v}(w_0)$ at some point $w_0\in(b,w_s)$. Let $p\in (D^-\hat{v}(w_0), D^+\hat{v}(w_0))$ and $\epsilon>0$. The function 
\[\varphi(w)=\hat{v}(w_0)+p(w-w_0)+\frac{1}{2\epsilon}(w-w_0)^2\]
satisfies $\hat{v}-\varphi$ has a local minimum at $w_0$. By supersolution property of $\hat{v}$, we get
\[G(w_0, \varphi(w_0), \varphi'(w_0), \varphi''(w_0))=\lambda \hat{v}(w_0)\ln\hat{v}(w_0)+R\epsilon p^2-(rw_0-c)p\geq 0.\]
Since $\epsilon$ is arbitrary, we get $\lambda \hat{v}(w_0)\ln\hat{v}(w_0)-(rw_0-c)p\geq 0$. In view of Lemma \ref{fv}, we must have $\lambda \hat{v}(w_0)\ln\hat{v}(w_0)-(rw_0-c)p=0$. But this cannot hold for every $p$. So the subdifferential at every point must be a singleton.

Since $\hat{v}$ is convex on $[b,w_s]$, to extend $C^1$-regularity up to the boundary, we only need to check $D^+\hat{v}(b)>-\infty$ and $D^-\hat{v}(w_s)<\infty$. We have already seen in the proof of Lemma \ref{vhat_convex_monotone} that $D^-\hat{v}(w_s)=0$. To bound $D^+\hat{v}(b)$ from below, we make use of the derivative of $\psi_{0}$. Simply observe that $D^+\hat{v}(b)=\varepsilon D^+\hat{u}(b)\hat{v}(b)$, and
\[D^+\hat{u}(b)=\lim_{w\rightarrow b+} \frac{\hat{u}(w)-1}{w-b}\geq \lim_{w\rightarrow b+}\frac{\psi_{0}(w)-1}{w-b}=D^+\psi_{0}(b)>-\infty.\]
\end{proof}

\begin{prop}\label{vhatC2}
$\hat{v}\in C^2[b,w_s)$ and satisfies $\hat{v}'<0$ and $\hat{v}''>0$ in $[b,w_s).$\footnote{The derivatives at $w=b$ is understood to be the continuous extension of interior derivatives.} In addition, $\hat{v}$ solves the second order equation
\begin{equation}\label{vODE}
\lambda v\ln v=-R\frac{(v')^2}{v''}+(rw-c)v', \quad w\in(b,w_s).
\end{equation}
\end{prop}
\begin{proof}
$\hat{v}'<0$ is due to Lemma \ref{vhat_convex_monotone}. Let $f(w):=(rw-c)\hat{v}'(w)-\lambda \hat{v}(w)\ln\hat{v}(w)$. By Lemmas \ref{fv} and \ref{vhatC1}, $f$ is continuous, non-negative and a.e. strictly positive in $(b,w_s)$. Let $g(w):=R(\hat{v}'(w))^2/f(w)$. The proof of $C^2$-regularity consists of two steps.

Step 1. Show that for any interval $[w_1, w_2]\subset [b,w_s]$ such that $f>0$ on $[w_1, w_2]$, $\hat{v}\in C^2[w_1,w_2]$. Notice that $g$ is continuous on $[w_1, w_2]$.

First of all, we show $\hat{v}$ is a viscosity solution of 
\begin{equation}\label{vhatC2_eq1}
-v''(w)+g(w)=0, \quad w\in(w_1, w_2).
\end{equation}
Let $w_0\in(w_1,w_2)$ and $\varphi\in C^2(w_1,w_2)$ be any test functions such that $\hat{v}-\varphi$ has a local maximum at $w_0$. Since $\hat{v}$ is a $C^1$ subsolution of $G=0$, we have $\varphi'(w_0)=\hat{v}'(w_0)$ and
\[G(w_0, \hat{v}(w_0), \hat{v}'(w_0), \varphi''(w_0))\leq 0.\]
Since $\hat{v}'(w_0)<0$, we must have $\varphi''(w_0)>0$ for the above $G$ to be finite. Writing out the expression for $G$ and optimizing over $\pi$, we get
\[-f(w_0)+R\frac{(\hat{v}'(w_0))^2}{\varphi''(w_0)}=\lambda \hat{v}(w_0)\ln \hat{v}(w_0)+R\frac{(\hat{v}'(w_0))^2}{\varphi''(w_0)}-(rw_0-c)\hat{v}'(w_0)\leq 0,\]
which, after multiplying by the positive quantity $\frac{\varphi''(w_0)}{f(w_0)}$, is precisely
\[-\varphi''(w_0)+g(w_0)\leq 0.\]
This shows $\hat{v}$ is a subsolution of \eqref{vhatC2_eq1}. Let $w_0\in(w_1,w_2)$ and $\varphi\in C^2(w_1,w_2)$ be any test function such that $\hat{v}-\varphi$ has a local minimum at $w_0$. If $\varphi''(w_0)\leq 0$, then we immediately have $-\varphi''(w_0)+g(w_0)\geq 0$ since $g$ is nonnegative. If $\varphi''(w_0)> 0$, then we use $\hat{v}$ is a $C^1$ supersolution of $G=0$ to obtain $\varphi'(w_0)=\hat{v}'(w_0)$ and
\[G(w_0, \hat{v}(w_0), \hat{v}'(w_0), \varphi''(w_0))\geq 0.\]
Optimizing over $\pi$ in the expression for $G$, we also get $-\varphi''(w_0)+g(w_0)\geq 0$. This shows $\hat{v}$ is a supersolution of \eqref{vhatC2_eq1}.

Next, we follow the argument on page 652 of \cite{ShreveSoner94} and consider the Poisson equation
\begin{equation}\label{vhatC2_eq2}
-v''+g=\epsilon
\end{equation}
with Dirichlet boundary conditions $v(w_1)=\hat{v}(w_1)$, $v(w_2)=\hat{v}(w_2)$. Here $\epsilon$ is a real number of our choice. We can integrate $g-\epsilon$ twice to get a $C^2[w_1, w_2]$ solution, denoted by $v_\epsilon$.
To compare $\hat{v}$ with $v_\epsilon$, first take $\epsilon>0$ and suppose $\hat{v}-v_\epsilon$ has a local maximum at some point $w_0\in(w_1, w_2)$. Since $\hat{v}$ is a viscosity subsolution of \eqref{vhatC2_eq1}, we have
\[-v_\epsilon''(w_0)+g(w_0)\leq 0,\]
which contradicts \eqref{vhatC2_eq2}. So the maximum must be attained on the boundary where it is zero. This means $\hat{v}\leq v_\epsilon$. Letting $\epsilon\rightarrow 0$ yields $\hat{v}\leq v_0$. The reverse inequality is obtained by taking $\epsilon<0$ and using $\hat{v}$ is a viscosity supersolution of \eqref{vhatC2_eq1}. This finishes the proof that $\hat{v}=v_0\in C^2[w_1,w_2]$.

Step 2. Show $f(w)>0$ for any $w\in[b,w_s)$. 

We use an argument similar to that on page 811-812 of \cite{Janecek12}. Pick any point $w_1\in (b,w_s)$ where $f(w_1)>0$. Since $f$ is continuous, $f>0$ in a neighborhood of $w_1$. Suppose $f$ vanishes at some point to the left of $w_1$. Let $w_0:=\sup\{w\in[b,w_1): f(w_0)=0\}$. By step 1, $\hat{v}$ satisfies equation \eqref{vhatC2_eq1} in the classical sense in $(w_0, w_1)$. Let $w\in(w_0,w_1)$. By mean value theorem,
\begin{equation}\label{vhatC2_eq3}
\frac{f(w)-f(w_0)}{w-w_0}=f'(z)=(r-\lambda)\hat{v}'(z)+(rz-c)\hat{v}''(z)-\lambda\hat{v}'(z)\ln\hat{v}(z)
\end{equation}
for some $z\in(w_0,w)$. Let $w\rightarrow w_0+$. Notice that $\hat{v}''(z)\rightarrow \infty$ because $\hat{v}''(z)=g(z)$ from equation \eqref{vhatC2_eq1}, and $g(z)$ has a strictly positive numerator and a denominator that is going to zero from the positive side. So the middle term on the right hand side of \eqref{vhatC2_eq3} is exploding to $-\infty$ while the other two terms converge to finite numbers. This contradicts the non-negativity of the left hand side. So $f(w_1)>0$ necessarily implies $f(w)>0$ for all $w\in[b,w_1)$. Since $f>0$ a.e., we conclude that $f>0$ in $[b,w_s)$. Combining step 1 and 2, we have $\hat{v}\in C^2[b,w_s)$.

From the proof of Lemma \ref{fv}, we know $\hat{v}''>0$ in $(b,w_s)$. Optimizing over $\pi$ in \eqref{P2:a} leads to \eqref{vODE}. Since $\hat{v}'(b)<0$, \eqref{vODE} implies $\hat{v}''(b)>0$. 
\end{proof}

Once we have $C^2$-regularity, we can further upgrade to infinite differentiability with little effort.
\begin{cor}\label{vhatCinfty}
$\hat{v}\in C^{\infty}[b,w_s)$.
\end{cor}
\begin{proof}
Let $g$ be defined as before. With $\hat{v}\in C^2[b,w_s)$, we now have $g\in C^1[b,w_s)$. It then follows from $\hat{v}''=g$ that $\hat{v}\in C^3[b,w_s)$. This in turn implies $g\in C^2[b,w_s)$ and so on. Inductively, we will get $\hat{v}\in C^\infty[b,w_s)$.
\end{proof}
\begin{remark}
Since $f(w_s)=0$, only $C^1$-regularity is guaranteed at the right boundary. Even in the non-robust case, it is possible to have an unbounded second derivative at the safe level.
\end{remark}

Going back to the original problem through $\hat{u}=\frac{1}{\varepsilon}\ln \hat{v}$, we have the following proposition for $\hat{u}$.
\begin{prop}\label{uhat_reg}
$\hat{u}\in C^1[b,w_s]\cap C^{2}[b,w_s)$, and satisfies $\hat{u}'<0$ and $\varepsilon(\hat{u}')^2+\hat{u}''>0$ in $[b,w_s)$. In addition, $\hat{u}$ solves the second order equation
\begin{equation}\label{uODE}
\lambda u=-R\frac{(u')^2}{\varepsilon (u')^2+u''}+(rw-c)u', \quad w\in(b,w_s).
\end{equation}
\end{prop}

\section{Verification}\label{sec:verification}

In order to relate $\hat{u}$ to the value function through verification, we first need to show the feedback forms lead to a pair of admissible controls under which the SDE for the controlled wealth process has a unique strong solution. The $\pi$ attaining the infimum in $F(w,\hat{u}, \hat{u}', \hat{u}'')$ is given by
\[\pi^\ast=-\frac{\mu-r}{\sigma^2}\frac{\hat{u}'}{\varepsilon(\hat{u}')^2+\hat{u}''}=-\frac{\mu-r}{\sigma^2}\frac{\hat{v}'}{\hat{v}''},\]
which is the same as the $\pi$ attaining the infimum in $G(w,\hat{v}, \hat{v}', \hat{v}'')$. We already know from the previous section that $\pi^\ast$ is smooth in $(b,w_s)$, thus locally Lipschitz. We will show $\pi^\ast$ is also well-behaved near the boundary.

\begin{lemma}\label{pi_bdd}
\begin{equation*}
0<\pi^\ast(w)< \frac{2(c-rw)}{\mu-r}, \quad w\in[b,w_s).
\end{equation*}
\end{lemma}
\begin{proof}
The lower bound is trivial. For the upper bound, rewrite equation \eqref{uODE} as
\begin{equation}\label{pi_eq_u}
\lambda \hat{u}=\left(\frac{\mu-r}{2}\pi^\ast+rw-c\right)\hat{u}'.
\end{equation}
For $w\in[b,w_s)$, since $\hat{u}(w)>0$ and $\hat{u}'(w)<0$, we must have $\frac{\mu-r}{2}\pi^\ast(w)+rw-c<0$.
\end{proof}

\begin{cor}
$\theta^\ast:=\sigma\varepsilon \pi^\ast\hat{u}'$ is bounded and satisfies $\lim_{w\rightarrow w_s-}\theta^\ast(w)=0$. More precisely,
\[\frac{2\sigma\varepsilon}{\mu-r}(c-rb)\hat{u}'(b)\leq \frac{2\sigma\varepsilon}{\mu-r}(c-rw)\hat{u}'(w)<\theta^\ast(w)< 0, \quad w\in[b,w_s).\]
\end{cor}
It will be verified later that $\pi^\ast(w)$, $\theta^\ast(w)$ are the optimal controls for $w\in(b,w_s)$. Observe that the upper bound for $\pi^\ast$ given by Lemma \ref{pi_bdd} is $\varpi$. In fact, we can tighten the bound to $\pi_0$ and show $\pi^\ast$ is non-increasing with respect to $\varepsilon$. 

\begin{prop}\label{pihat_eps}
$\pi^\ast(w;\varepsilon)$ is non-increasing in $\varepsilon$ for $\varepsilon\geq 0$. In particular, $\pi^\ast(w;\varepsilon)\leq \pi_0(w)$.
\end{prop}
\begin{proof}
Let $0<\varepsilon_1<\varepsilon_2$ and write $\hat{u}_i(w)$ for $\hat{u}(w;\varepsilon_i)$, $i=1,2$. $\hat{v}_i$ and $\pi^\ast_i$ are defined similarly. First of all, by comparison principle for problem \eqref{P2}, we have $\hat{v}_1\leq \hat{v}_2$ and thus $\hat{u}_1\leq \hat{u}_2$. Since $\hat{u}_1(b)=\hat{u}_2(b)=1$, we deduce
\[\hat{u}'_1(b)=\lim_{w\rightarrow b+}\frac{\hat{u}_1(w)-1}{w-b}\leq \lim_{w\rightarrow b+}\frac{\hat{u}_2(w)-1}{w-b}=\hat{u}'_2(b).\] 
Let $w\rightarrow b+$ in equation \eqref{pi_eq_u}, we see that 
\[\lambda=\left(\frac{\mu-r}{2}\pi^\ast(b;\varepsilon)+rb-c\right)\hat{u}'(b;\varepsilon).\]
Since $\hat{u}'_1(b)\leq \hat{u}'_2(b)<0$, we must have $\pi^\ast_1(b)\geq \pi^\ast_2(b)$. By Lemma \ref{pi_bdd}, we also have $\pi^\ast_1(w_s)= \pi^\ast_2(w_s)=0$.\footnote{By $\pi^\ast(w_s)$, we mean $\lim_{w\rightarrow w_s-}\pi^\ast(w)$ since $\hat{v}''(w)$ may not exist at $w=w_s$.} Claim that $\pi^\ast_1(w)\geq \pi^\ast_2(w)$ for all $w\in[b,w_s]$.

From equation \eqref{vODE}, we obtain
\begin{equation*}
\lambda \hat{v}\ln\hat{v}=\frac{\mu-r}{2}\hat{v}'\pi^\ast+(rw-c)\hat{v}', \quad w\in(b,w_s).
\end{equation*}
By Corollary \ref{vhatCinfty}, we can differentiate the above equation. After rearranging terms, we get
\begin{equation}\label{Dpi_eq_v}
\frac{\mu-r}{2}(\pi^\ast)'=R+\lambda-r+\frac{\mu-r}{\sigma^2}\frac{rw-c}{\pi^\ast}+\lambda\ln\hat{v}, \quad w\in(b,w_s).
\end{equation}
Suppose on the contrary, $\pi^\ast_1-\pi^\ast_2$ attains negative minimum at a point $w_0\in(b,w_s)$. By first order condition, we have $(\pi^\ast_1)'(w_0)=(\pi^\ast_2)'(w_0)$. Equation \eqref{Dpi_eq_v} then yields the contradiction:
\[0=\frac{\mu-r}{\sigma^2}(rw_0-c)\frac{\pi^\ast_1(w_0)-\pi^\ast_2(w_0)}{\pi^\ast_1(w_0)\pi^\ast_2(w_0)}+\lambda(\ln\hat{v}_2(w_0)-\ln\hat{v}_1(w_0))>0,\]
where we used $\pi^\ast_i>0$, $\hat{v}_2\geq\hat{v}_1$ in $(b,w_s)$, and the assumption $\pi^\ast_1(w_0)-\pi^\ast_2(w_0)<0$. Therefore, the claim holds. 

If $\varepsilon_1=0$, i.e. $\pi^\ast_1=\pi_0$, then simple computation shows $\pi^\ast_1$ satisfies \eqref{Dpi_eq_v} with $\hat{v}_1:=1$. Exactly the same comparison argument implies $\pi^\ast_1\geq\pi^\ast_2$ everywhere on $[b,w_s]$. 
\end{proof}

\begin{prop}\label{pi_Lip}
$\pi^\ast$ is Lipschitz continuous in $(b,w_s)$ and satisfies
\begin{equation}\label{tangent_line_at_ws}
\lim_{w\rightarrow w_s-}(\pi^\ast)'(w)=-\frac{\mu-r}{\sigma^2(d-1)}=\pi_0'(w_s).
\end{equation}
\end{prop}
\begin{proof}
For Lipschitz continuity, it suffices to show $\pi^\ast$ has bounded first derivative in $(b,w_s)$. Since $\pi^\ast>0$ in $[b,w_s)$, equation \eqref{Dpi_eq_v} implies $(\pi^\ast)'$ is bounded on any subset of $(b,w_s)$ that is away from $w_s$. It remains to show \eqref{tangent_line_at_ws}. Let $\ell:=\liminf_{w\rightarrow w_s-}(\pi^\ast)'(w)$ and $L:=\limsup_{w\rightarrow w_s-}(\pi^\ast)'(w)$. By Proposition \ref{pihat_eps}, we have
\[\frac{rw-c}{\pi^\ast}\leq \frac{rw-c}{\pi_0}=-\frac{\sigma^2r(d-1)}{\mu-r}.\]
The above inequality and \eqref{Dpi_eq_v} imply
\begin{equation*}
\frac{\mu-r}{2}(\pi^\ast)'\leq R+\lambda-rd+\lambda\ln\hat{v}.
\end{equation*}
Simple algebra shows $R+\lambda-rd=R/(1-d)<0$. Since $\hat{v}(w)\rightarrow 1$ as $w\rightarrow w_s-$, we know $(\pi^\ast)'$ is negative and bounded away from zero near $w_s$. In particular, the limit superior $L$ satisfies
\begin{equation*}
\frac{\mu-r}{2}L\leq R+\lambda-rd=\frac{R}{1-d}<0.
\end{equation*}
Now, apply generalized l'H\^{o}pital's rule \cite[Theorem II]{Taylor52} to \eqref{Dpi_eq_v}. We deduce
\[\frac{\mu-r}{2}\ell\geq R+\lambda-r+\frac{\mu-r}{\sigma^2}\liminf_{w\rightarrow w_s-}\frac{r}{(\pi^\ast)'(w)}=R+\lambda-r+\frac{\mu-r}{\sigma^2}\frac{r}{L}.\]
This leads to a chain of inequalities which is in fact a chain of equalities:
\begin{equation*}
0>\frac{R}{1-d}\geq \frac{\mu-r}{2}L \geq\frac{\mu-r}{2}\ell \geq R+\lambda-r+r(1-d)=\frac{R}{1-d}.
\end{equation*}
So we have proved $\ell=L=\frac{2R}{(\mu-r)(1-d)}=-\frac{\mu-r}{\sigma^2(d-1)}$.
\end{proof}

We are now ready to prove the verification theorem. For any $C^2$ function $\varphi$ and $\pi, \theta\in\mathbb{R}$, define
\[\mathcal{L}^{\pi, \theta}\varphi(w):=[rw-c+(\mu+\sigma\theta-r)\pi]\varphi'+\frac{1}{2}\sigma^2\pi^2 \varphi''(w).\]
\begin{thm}[Verification theorem]
Suppose $u: [b,\infty)\rightarrow [0,1]$, $\Pi:[b,\infty)\rightarrow \mathbb{R}$ and $\Theta:\mathbb{R}\times [b,\infty)\rightarrow \mathbb{R}$ are measurable functions satisfying the following conditions:
\begin{enumerate}[label=(\roman{*}), ref=(\roman{*})]
\item $u\in C^1[b,w_s]\cap C^2[b,w_s)$; \label{cond:i}
\item $u$ is a solution of 
\begin{equation}\label{eq_infsup}
\lambda u(w)=\inf_\pi \sup_\theta \left\{-\frac{1}{2\varepsilon} \theta^2+\mathcal{L}^{\pi,\theta}u(w)\right\}, \ w\in(b,w_s);
\end{equation} 
\label{cond:ii}
\item $u(b)=1$ and $u(w)=0$ for $w\geq w_s$; \label{cond:iii}
\item $\Pi(w)$ attains the infimum in (ii) for each $w\in(b,w_s)$; $\Theta(\pi, w)$ attains the supremum in \ref{cond:ii} for each $\pi\in\mathbb{R}$ and $w\in(b,w_s)$; \label{cond:iv}
\item $\Pi(w)=\Theta(\pi, w)=0$ if $w\notin (b,w_s)$; \label{cond:v}
\item $\Pi$ is bounded and Lipschitz continuous in $(b,w_s)$; $\Theta$ is bounded on $[\pi_1,\pi_2]\times [b,\infty)$ for any compact interval $[\pi_1,\pi_2]\subset\mathbb{R}$. 
\label{cond:vi}
\end{enumerate}
Then $\psi=u$ on $[b,\infty)$, and $\Pi(\cdot), \Theta(\Pi(\cdot),\cdot)$ are optimal Markovian controls.   
\end{thm}
\begin{proof}
Same as \cite{BayraktarYoung07b}, we let $\Delta$ be the ``coffin state'' and $[b,\infty)\cup\{\Delta\}$ be the one point compactification of $[b,\infty)$. Define the extension of $u$ to $[b,\infty)\cup \{\Delta\}$ by assigning $u(\Delta)=0$.

1. Let $w>b$. By conditions \ref{cond:v} and \ref{cond:vi}, the SDE
\begin{equation*}
dW_t=[rW_t+(\mu-r)\Pi(W_t)-c]dt+\sigma \Pi(W_t) dB_t, \quad W_0=w
\end{equation*}
has a unique strong solution $W^{w,\Pi}$ w.r.t. the filtered probability space $(\Omega,\mathcal{F},\mathbb{F},\mathbb{P})$. Let $\bfpi^\ast_t:=\Pi(W^{w,\Pi}_t)$ and write $W^{w,\bfpi^\ast}:=W^{w,\Pi}$. $\bfpi^\ast\in\mathscr{A}$ since $\Pi$ is bounded and measurable. Define $\tau^\ast_b:=\inf\{t\geq 0: W^{w,\bfpi^\ast}\leq b\}$ and $\tau^\ast:=\inf\{t\geq 0: W^{w,\bfpi^\ast}\geq w_s\}\wedge \tau_d$. 

Let $\mathbb{Q}\in\mathscr{Q}$ be any candidate measure with corresponding drift distortion process $\bftheta$. By Girsanov theorem, $B_t^{\mathbb{Q}}:=B_t-\int_0^t \bftheta_s ds$ is a $\mathbb{Q}$-Brownian motion. $W^{w,\bfpi^\ast}$ satisfies
\begin{equation*}
dW_t=[rW_t+(\mu+\sigma\bftheta_t-r)\bfpi^\ast_t-c]dt+\sigma \bfpi^\ast_t dB^{\mathbb{Q}}_t, \quad W_0=w.
\end{equation*}
Recall that $\tau_d$ is the first jump time of the $\mathbb{P}$-Poisson process $N$ with rate $\lambda$ that is independent of $\mathbb{F}$. The definition of $\mathscr{Q}$ ensures that $N$ is also a $\mathbb{Q}$-Poisson process with the same rate. Let $\ubar{W}^{w,\bfpi^\ast}_t:=W_t^{w,\bfpi^\ast}1_{\{t<\tau_d\}}+\Delta 1_{\{t\geq \tau_d\}}$. 
$\ubar{W}^{w,\bfpi^\ast}$ is a progressively measurable process in the enlarged filtration $\mathbb{H}$ which includes information generated by $N$. 
It is easy to see that $\ubar{W}^{w,\bfpi^\ast}$ satisfies:
\[dW_t=[rW_t+(\mu+\sigma\bftheta_t-r)\bfpi^\ast_t-c]dt+\sigma\bfpi^\ast_t dB^{\mathbb{Q}}_t-(\Delta-W_{t-})dN_t, \quad W_0=w.\]
Applying It\^{o}'s lemma to $u(\ubar{W}_t^{w,\bfpi^\ast})$ and using that $u(\Delta)=0$, we have
\begin{equation*}
\begin{aligned}
u(\ubar{W}^{w,\bfpi^\ast}_{\tau^\ast_b\wedge\tau^\ast})
&=u(w)+\int_0^{\tau^\ast_b\wedge\tau^\ast} \mathcal{L}^{\bfpi^\ast_s,\bftheta_s}u(W^{w,\bfpi^\ast}_s)-\lambda u(W^{w,\bfpi^\ast}_{s}) ds\\
&\quad +\int_0^{ \tau^\ast_b\wedge \tau^\ast}u'(W^{w,\bfpi^\ast}_s)\sigma \bfpi^\ast_s dB^{\mathbb{Q}}_s-u(W^{w,\bfpi^\ast}_{s-})d(N_s-\lambda s).
\end{aligned}
\end{equation*}
Since $u$, $u'$ and $\Pi$ are bounded on $[b,w_s]$, the It\^{o} integral vanishes upon taking $\mathbb{Q}$-expectation and we get
\[\mathbb{E}^{\mathbb{Q}}\left[u(\ubar{W}^{w,\bfpi^\ast}_{\tau^\ast_b\wedge\tau^\ast})\right]=u(w)+\mathbb{E}^{\mathbb{Q}}\left[\int_0^{\tau^\ast_b\wedge\tau^\ast} \mathcal{L}^{\bfpi^\ast_s,\bftheta_s}u(W^{w,\bfpi^\ast}_s)-\lambda u(W^{w,\bfpi^\ast}_{s}) ds\right].
\]
Conditions \ref{cond:ii}, \ref{cond:iv} and that $\bfpi^\ast=\Pi(W^{w,\bfpi^\ast})$ imply for $0\leq s<\tau_b^\ast\wedge \tau^\ast$,
\begin{equation*}
\begin{aligned}
0&=\inf_\pi \sup_\theta \left\{-\frac{1}{2\varepsilon} \theta^2+\mathcal{L}^{\pi,\theta}u(W^{w,\bfpi^\ast}_s)\right\}-\lambda u(W^{w,\bfpi^\ast}_s)\\
&=\sup_\theta \left\{-\frac{1}{2\varepsilon} \theta^2+\mathcal{L}^{\bfpi^\ast_s,\theta}u(W^{w,\bfpi^\ast}_s)\right\}-\lambda u(W^{w,\bfpi^\ast}_s)\\
&\geq -\frac{1}{2\varepsilon}\bftheta_s^2+\mathcal{L}^{\bfpi^\ast_s,\bftheta_s}u(W^{w,\bfpi^\ast}_s)-\lambda u(W^{w,\bfpi^\ast}_s).
\end{aligned}
\end{equation*}
So we have
\begin{equation*}
\mathbb{E}^{\mathbb{Q}}\left[u(\ubar{W}^{w,\bfpi^\ast}_{\tau^\ast_b\wedge \tau^\ast})\right]\leq u(w)+\mathbb{E}^{\mathbb{Q}}\left[ \int_0^{\tau^\ast_b\wedge\tau^\ast} \frac{1}{2\varepsilon}\bftheta^2_s ds\right].
\end{equation*}
Equivalently,
\begin{equation*}
u(w)\geq \mathbb{E}^{\mathbb{Q}}\left[u(\ubar{W}^{w,\bfpi^\ast}_{\tau^\ast_b\wedge \tau^\ast})-\int_0^{\tau^\ast_b\wedge\tau^\ast} \frac{1}{2\varepsilon}\bftheta^2_s ds\right].
\end{equation*}
By condition \ref{cond:v}, $W^{w,\bfpi^\ast}$ will stay constant once it reaches the safe level. This means, if the safe level is reached, then death will definitely occur before ruin. So we have $\{\tau^\ast_b<\tau^\ast\}=\{\tau^\ast_b<\tau^\ast_d\}$. Since $u(\ubar{W}^{w,\bfpi^\ast}_{\tau^\ast_b\wedge \tau^\ast})=1_{\{\tau^\ast_b<\tau^\ast\}}=1_{\{\tau^\ast_b<\tau_d\}}$ and $\tau^\ast_b\wedge \tau^\ast\leq \tau_d$, we get
\begin{equation*}
u(w)\geq \mathbb{E}^{\mathbb{Q}}\left[1_{\{\tau^\ast_b<\tau_d\}}-\int_0^{\tau_d} \frac{1}{2\varepsilon}\bftheta^2_s ds\right].
\end{equation*}
This holds for all $\mathbb{Q}\in\mathscr{Q}$. So
\[u(w)\geq \sup_{\mathbb{Q}\in\mathscr{Q}}\mathbb{E}^{\mathbb{Q}}\left[1_{\{\tau^\ast_b<\tau_d\}}-\int_0^{\tau_d} \frac{1}{2\varepsilon}\bftheta^2_s ds\right]\geq \inf_{\bfpi\in\mathscr{A}}\sup_{\mathbb{Q}\in\mathscr{Q}}\mathbb{E}^{\mathbb{Q}}\left[1_{\{\tau^{w,\bfpi}_b<\tau_d\}}-\frac{1}{\varepsilon}\int_0^{\tau_d} \frac{1}{2}\bftheta^2_s ds\right]=\psi(w),\]
where we put superscripts on $\tau_b$ in the last step to indicate its dependence on the initial wealth and the control.

2. Let $\bfpi\in\mathscr{A}$ be any admissible investment strategy and $W^{w,\bfpi}$ be the solution to the SDE:
\begin{align*}
dW_t&=[rW_t+(\mu-r)\bfpi_t-c]dt+\sigma\bfpi_t dB_t, \quad W_0=w.
\end{align*}
Let $\bftheta^\ast_t:=\Theta(\bfpi_t,W^{w,\bfpi}_t)$. $\bftheta^\ast$ is $\mathbb{F}$-progressively measurable since both $\bfpi$ and $W^{w,\bfpi}$ are, and $\Theta$ is a measurable function. Since $\bfpi_t$ is a.s. bounded uniformly in $t$, condition \ref{cond:vi} ensures $\bftheta^\ast$ satisfies all integrability conditions in the definition of $\mathscr{Q}$. So there exists a measure $\mathbb{Q}^\ast\in\mathscr{Q}$ satisfying $\frac{d\mathbb{Q}^\ast_t}{d\mathbb{P}_t}=\mathcal{E}(\int_0^t\bftheta^\ast_s dB_s)$ where $\mathcal{E}$ denotes the stochastic exponential. It follows from Girsanov theorem that $B^{\mathbb{Q}^\ast}_t:=B_t-\int_0^t \bftheta^\ast_s ds$ is a $\mathbb{Q}^\ast$-Brownian motion. So $W^{w,\bfpi}$ satisfies:
 \[dW_t=[rW_t+(\mu+\sigma\bftheta^\ast_t-r)\bfpi_t-c]dt+\sigma\bfpi_t dB^{\mathbb{Q}^\ast}_t, \quad W_0=w.\]
 Define $\tau_b:=\inf\{t\geq 0: W^{w,\bfpi}_t\leq b\}$ and $\tau:=\inf\{t\geq 0: W^{w,\bfpi}_t\geq w_s\}\wedge \tau_d$. Same as before, we work with the larger filtration $\mathbb{H}$ and consider the process $\ubar{W}^{w,\bfpi}=W_t^{w,\bfpi}1_{\{t<\tau_d\}}+\Delta 1_{\{t\geq \tau_d\}}$ which satisfies the SDE:
\[dW_t=[rW_t+(\mu+\sigma\bftheta^\ast_t-r)\bfpi_t-c]dt+\sigma\bfpi_t dB^{\mathbb{Q}^\ast}_t+(\Delta-W_{t-})dN_t,\quad W_0=w.\]
Again, thanks to the drift distortion $\bftheta^\ast$ being $\mathbb{F}$-adapted, $N$ remains a Poisson process with rate $\lambda$ under $\mathbb{Q}^\ast$. By It\^{o}'s lemma and that $u(\Delta)=0$, we have for any $t\geq 0$,
\begin{align*}
u(\ubar{W}^{w,\bfpi}_{\tau_b\wedge\tau\wedge t})
&=u(w)+\int_0^{\tau_b\wedge\tau\wedge t} -\lambda u(W^{w,\bfpi}_{s})+\mathcal{L}^{\bfpi_s,\bftheta^\ast_s}u(W^{w,\bfpi}_s) ds\\
&\quad +\int_0^{ \tau_b\wedge \tau\wedge t}u'(W^{w,\bfpi})\sigma \bfpi_s dB^{\mathbb{Q}^\ast}_s-u(W^{w,\bfpi}_{s-})d(N_s-\lambda s).
\end{align*}
Taking $\mathbb{Q}^\ast$ expectation yields
\begin{equation*}
\mathbb{E}^{\mathbb{Q}^\ast}\left[u(\ubar{W}^{w,\bfpi}_{\tau_b\wedge \tau\wedge t})\right]= u(w)+\mathbb{E}^{\mathbb{Q}^\ast}\left[ \int_0^{\tau_b\wedge\tau\wedge t}-\lambda u(W^{w,\bfpi}_{s}) +\mathcal{L}^{\bfpi_s,\bftheta^\ast_s}u(W^{w,\bfpi}_s)ds\right],
\end{equation*}
where the It\^{o} integral vanishes because $u$, $u'$ are bounded and $\bfpi$ is $\mathbb{Q}^\ast_t$-a.s. bounded for all $t\geq 0$. By conditions \ref{cond:ii}, \ref{cond:iv}, and our definition of $\bftheta^\ast$, we know
\begin{equation}\label{verification_step2_eq1}
\begin{aligned}
0&=\inf_\pi \sup_\theta \left\{-\frac{1}{2\varepsilon} \theta^2+\mathcal{L}^{\pi,\theta}u(W^{w,\bfpi}_s)\right\}-\lambda u(W^{w,\bfpi}_s)\\
&\leq \sup_\theta \left\{-\frac{1}{2\varepsilon} \theta^2+\mathcal{L}^{\bfpi_s,\theta}u(W^{w,\bfpi}_s)\right\}-\lambda u(W^{w,\bfpi}_s)\\
&=-\frac{1}{2\varepsilon} (\bftheta^\ast_s)^2+\mathcal{L}^{\bfpi_s,\bftheta^\ast_s}u(W^{w,\bfpi}_s)-\lambda u(W^{w,\bfpi}_s)
\end{aligned}
\end{equation}
for $s\in[0,\tau_b\wedge \tau)$. So
\begin{equation}\label{verification_step2_eq2}
\mathbb{E}^{\mathbb{Q}^\ast}\left[u(\ubar{W}^{w,\bfpi}_{\tau_b\wedge \tau\wedge t})\right]\geq u(w)+\mathbb{E}^{\mathbb{Q}^\ast}\left[ \int_0^{\tau_b\wedge\tau\wedge t}\frac{1}{2\varepsilon}(\bftheta^\ast_s)^2 ds\right].
\end{equation}
Letting $t\rightarrow \infty$ and using bounded and monotone convergence theorems, we get
\begin{equation}\label{verification_step2_eq2b}
\mathbb{E}^{\mathbb{Q}^\ast}\left[u(\ubar{W}^{w,\bfpi}_{\tau_b\wedge \tau})\right]\geq u(w)+\mathbb{E}^{\mathbb{Q}^\ast}\left[ \int_0^{\tau_b\wedge\tau}\frac{1}{2\varepsilon}(\bftheta^\ast_s)^2 ds\right].
\end{equation}
Since $u(\ubar{W}^{w,\bfpi}_{\tau_b\wedge \tau})=1_{\{\tau_b<\tau\}}\leq 1_{\{\tau_b<\tau_d\}}$, we obtain
\begin{equation}\label{verification_step2_eq3}
u(w)\leq \mathbb{E}^{\mathbb{Q}^\ast}\left[1_{\{\tau_b<\tau_d\}}-\frac{1}{\varepsilon}\int_0^{\tau_b\wedge\tau}\frac{1}{2}(\bftheta^\ast_s)^2 ds\right].
\end{equation}
Let us first assume $\bfpi$ is an admissible strategy such that $\bfpi=0$ once ruin occurs or safe level is reached, so that the wealth process will stay at the ruin or safe level until death time. Denote by $\mathscr{A}_0$ the collection of such subclass of strategies. Then by condition \ref{cond:v}, we have $\bftheta^\ast_s=0$ for $\tau_b\wedge \tau<s<\tau_d$. Hence
\begin{equation*}
u(w)\leq \mathbb{E}^{\mathbb{Q}^\ast}\left[1_{\{\tau_b<\tau_d\}}-\frac{1}{\varepsilon}\int_0^{\tau_d}\frac{1}{2}\bftheta^\ast_s ds\right]\leq \sup_{\mathbb{Q}\in\mathscr{Q}} \mathbb{E}^{\mathbb{Q}}\left[1_{\{\tau_b<\tau_d\}}-\frac{1}{\varepsilon}\int_0^{\tau_d}\frac{1}{2}(\bftheta_s)^2 ds\right].
\end{equation*}
This holds for any $\bfpi\in\mathscr{A}_0$. So we have
\begin{equation*}
u(w)\leq \inf_{\bfpi\in\mathscr{A}_0}\sup_{\mathbb{Q}\in\mathscr{Q}} \mathbb{E}^{\mathbb{Q}}\left[1_{\{\tau^{w,\bfpi}_b<\tau_d\}}-\frac{1}{\varepsilon}\int_0^{\tau_d}\frac{1}{2}(\bftheta_s)^2 ds\right],
\end{equation*}
where we added superscripts to $\tau_b$ to indicate its dependence on the initial wealth and the control.
It remains to note that controls in $\mathscr{A}\backslash\mathscr{A}_0$ do not yield a smaller infimum because once ruin occurs, it becomes a history that cannot be altered; once safe level is reached, no policy can do a better job than zero ruin probability. Therefore, we actually have
\[u(w)\leq \inf_{\bfpi\in\mathscr{A}}\sup_{\mathbb{Q}\in\mathscr{Q}} \mathbb{E}^{\mathbb{Q}}\left[1_{\{\tau^{w,\bfpi}_b<\tau_d\}}-\frac{1}{\varepsilon}\int_0^{\tau_d}\frac{1}{2}(\bftheta_s)^2 ds\right]=\psi(w).\]

3. As in step 1, let $W^{w,\Pi}$ be the unique strong solution of 
\begin{align*}
dW_t&=[rW_t+(\mu-r)\Pi(W_t)-c]dt+\sigma\Pi(W_t) dB_t, \quad W_0=w,
\end{align*}
and $\bfpi^\ast:=\Pi(W^{w,\Pi})\in\mathscr{A}$. Let 
$\bftheta^\ast_t:=\Theta(\Pi(W^{w,\Pi}_t),W^{w,\Pi}_t)$.
Conditions \ref{cond:v}, \ref{cond:vi} and the measurability of $\Theta, \Pi$ ensures $\bftheta^\ast$ is a bounded, $\mathbb{F}$-progressively measurable process. So there is a measure $\mathbb{Q}^\ast\in\mathscr{Q}$ having $\bftheta^\ast$ as the corresponding drift distortion process.
Repeat the analysis in step 2 using controls $\bfpi^\ast$ and $\bftheta^\ast$. \eqref{verification_step2_eq1} through \eqref{verification_step2_eq2b} now hold with equality. For \eqref{verification_step2_eq3}, since $\bfpi^\ast$ and $\bftheta^\ast$ will both be zero and remain zero until death time once the ruin level or the safe level is reached, we have $\{\tau^\ast_b<\tau^\ast\}=\{\tau^\ast_b<\tau_d\}$ and 
\[\psi(w)=u(w)=\mathbb{E}^{\mathbb{Q}^\ast}\left[1_{\{\tau^\ast_b<\tau_d\}}-\int_0^{\tau_d} \frac{1}{2\varepsilon}(\bftheta^\ast_s)^2 ds\right],\]
where $\tau_b^\ast$ and $\tau^\ast$ denote the ruin time and the minimum of safe and death times, respectively, when the wealth starts at $w$ and is controlled by $\bfpi^\ast$. This proves the optimality of the feedback forms.
\end{proof}

\noindent \textbf{\textit{Proof of Theorem \ref{main_thm}.}} The functions 
\[u(w):=\hat{u}(w)1_{\{w\leq w_s\}}, \ \   \Pi(w):=\pi^\ast(w)1_{(b,w_s)}(w), \ \ \Theta(\pi, w):=\sigma\varepsilon\pi\hat{u}'(w)1_{(b,w_s)}(w)\]
satisfy all conditions of the verification theorem. (i) follows from Proposition \ref{uhat_reg}. (ii) and (iii) hold because $\hat{u}$ solves \eqref{P1} and $F=0$ is equivalent to \eqref{eq_infsup}. (iv) follows from first order conditions and the definition of $\pi^\ast$ (see the beginning of Section~\ref{sec:verification}). (v) is clear from the definition of $\Pi$ and $\Theta$. (vi) holds by Propositions \ref{pi_Lip} and \ref{uhat_reg}; the latter implies $\hat{u}'$ is bounded on $[b,w_s]$.
\hfill \qed
\begin{remark}
Verification can be carried out even if $\pi^\ast$ is only known to be locally Lipschitz continuous, because the optimally controlled wealth process actually never reaches the safe level (see Proposition \ref{never_reach_ws}). What Proposition \ref{pi_Lip} shows on top of the global Lipschitz continuity is that $\pi^\ast(w;\varepsilon)$ is tangent to $\pi_0(w)$ at $w=w_s$ for all $0<\varepsilon<\infty$.
\end{remark}

In the remaining sections, we will speak of $\pi^\ast$, $\theta^\ast$ given by \eqref{pistar} and \eqref{thetastar} as the optimal Markovian controls. It is understood that they are optimal in the interval $(b,w_s)$.

\section{Other properties of the value function and the optimal investment policy}\label{sec:properties}

Let us first summarize some properties of $\psi$ and $\pi^\ast$ that we have already seen.
\begin{itemize}
\item[(i)] $\psi\in C^1[b,w_s]\cap C^2[b,w_s)$ and is strictly decreasing on $[b,w_s]$;
\item[(ii)] $\psi$ is non-decreasing in $\varepsilon$, bounded from below by $\psi_0$ and from above by $\psi_\infty\wedge \mathfrak{p}$;
\item[(iii)] $0<\pi^\ast\leq \pi_0$ in $[b,w_s)$ and $\pi^\ast$ is non-increasing in $\varepsilon$;
\item[(iv)] $\pi^\ast$ is Lipschitz continuous in $(b,w_s)$ and is tangent to $\pi_0$ at the safe level.
\end{itemize} 
In this section, we prove two additional properties. The first one reveals how the concavity of $\psi$ depends on parameters. The second one addresses the question of whether the safe level can be reached by the optimally controlled wealth process. In the non-robust case, \cite{Young04} shows it is never reached in finite time. Same phenomenon exists for our robust problem; the individual either loses the game, or ``wins" the game by dying.

\begin{prop}\label{psiconvex} \ 
\begin{itemize}
\item[(i)] If $r\leq\lambda$, then $\psi$ is convex on $[b,w_s]$. If $r<\lambda$, $\psi$ is strictly convex.
\item[(ii)] If $r>\lambda$, then $\psi$ changes concavity at most once on $[b,w_s]$. If $0\leq \varepsilon\leq \frac{R}{rd-\lambda}$, $\psi$ is strictly convex on $[b,w_s]$. If $\varepsilon>\frac{R}{r-\lambda}$, $\psi$ is strictly concave in $[b,w_0)$ and strictly convex in $(w_0,w_s]$ where $w_0$ is the unique point in $(b,w_s)$ satisfying $(rw_0-c)\psi'(w_0)-\lambda \psi(w_0)=\frac{R}{\varepsilon}$.
\end{itemize}
\end{prop}
\begin{proof}
(i) When $\varepsilon=0$, strict convexity holds regardless of the sign of $r-\lambda$. Assume $\varepsilon>0$. Let $f(w):=(rw-c)\psi'-\lambda \psi$. When proving Proposition \ref{vhatC2}, we showed that $(rw-c)\hat{v}'-\lambda\hat{v}\ln\hat{v}>0$ in $[b,w_s)$. In terms of $\psi$ which equals $\hat{u}$ on $[b,w_s]$, we have $\varepsilon e^{\varepsilon \psi}[(rw-c)\psi'-\lambda\psi]>0$, which implies $f>0$ on $[b,w_s)$.
Recall that $\psi$ satisfies
\[R\frac{(\psi')^2}{\varepsilon (\psi')^2+\psi''}=f.\]
Moving $\psi''$ to one side and everything else to the other side, we obtain
\begin{equation}\label{psiconvex_eq1}
\psi''=\left(\frac{R}{f}-\varepsilon\right)(\psi')^2.
\end{equation}
We see that the sign of $\psi''$ depends on the relative size of $f$ to $R/\varepsilon$. Since $\psi\in C^2[b,w_s)$, we can differentiate $f$ and get
\begin{equation}\label{psiconvex_eq2}
f'=(r-\lambda)\psi'+(rw-c)\psi''\geq (rw-c)\psi'',
\end{equation}
where the inequality follows from $\psi'<0$ and the assumption that $r\leq \lambda$. Since $f(w_s)=0$, $f$ attains maximum either at an interior point or at $w=b$. In both cases, we have $f'(w_m)\leq 0$ where $w_m\in[b,w_s)$ is the point where maximum is attained. It follows from \eqref{psiconvex_eq2} that $\psi''(w_m)\geq 0$ and then from \eqref{psiconvex_eq1} that $f(w_m)\leq R/\varepsilon$. Since $w_m$ is a maximum point, $f(w)\leq R/\varepsilon$ for all $w\in[b,w_s]$. This in turn implies by \eqref{psiconvex_eq1} that $\psi''(w)\geq 0$ for all $w\in(b,w_s)$. Since $\psi$ is continuous, interior convexity can be extended to the boundary. If $r<\lambda$, then the inequality in \eqref{psiconvex_eq2} becomes strict and we have $\psi''(w_m)>0$ at the point $w_m$ of maximality of $f$. Subsequent inequalities all become strict and we obtain strict convexity of $\psi$.

(ii) First of all, equation \eqref{psiconvex_eq1} implies in any circumstances, regardless of the sign of $r-\lambda$ and the value of $\varepsilon$, $\psi$ will be strictly convex in a neighborhood of $w_s$. This is because $f(w_s)=0$ so that $R/f(w)-\varepsilon>0$ for $w$ sufficiently close to $w_s$. Let $r>\lambda$. Then \eqref{psiconvex_eq2} becomes
\begin{equation}\label{psiconvex_eq3}
f'=(r-\lambda)\psi'+(rw-c)\psi''< (rw-c)\psi'', \quad w\in[b,w_s).
\end{equation}
If $\psi$ changes concavity at $w_0$, then $\psi''(w_0)=0$ and the above inequality implies $f'(w_0)<0$. So $f$ is strictly decreasing whenever $\psi$ changes concavity.\footnote{$\psi$ cannot be locally linear because otherwise on one hand, \eqref{psiconvex_eq3} implies $f'<0$ and $f$ is locally strictly decreasing; on the other hand, \eqref{psiconvex_eq1} implies $f$ is locally constant.} Looking at \eqref{psiconvex_eq1}, we deduce that $\psi$ can only change from concave to convex if concavity changes at all. Since we have already argued $\psi$ is strictly convex in a neighborhood of $w_s$, we conclude that if $\psi$ is not convex everywhere, then it changes concavity only once; it is strictly concave up to the (unique) point $w_0$ where $f(w_0)=R/\varepsilon$ and is strictly convex afterwards. We also note that since $f$ can only touch or cross the horizontal line at $R/\varepsilon$ in a decreasing fashion, $f(w)>R/\varepsilon$ for $w\in[b,w_0)$ and $f(w)<R/\varepsilon$ for $w\in(w_0,w_s]$.

Next, we identify some cases when $\psi$ changes or does not change concavity. In view of the way $f$ intersect the horizontal line at $R/\varepsilon$, it suffices to check whether $f(b)>R/\varepsilon$. We have
\[f(b)=(rb-c)\psi'(b)-\lambda.\]

If $0\leq \varepsilon\leq \frac{R}{rd-\lambda}$, then $f(b)\leq (rb-c)\psi'_0(b)-\lambda=rd-\lambda\leq R/\varepsilon$ and there will be no concavity change for $\psi$. 

If $\varepsilon>\frac{R}{r-\lambda}$, then we consider two cases. If $\psi'(b)>\frac{1}{b-w_s}=\frac{r}{rb-c}$, then $\psi$ cannot be convex everywhere. If it is convex everywhere, it will stay above its tangent line passing through the point $(b,1)$. But the point $(w_s,0)$ lies below this tangent line, which means the right boundary condition is not satisfied. If $\psi'(b)\leq \frac{r}{rb-c}$, then $f(b)\geq r-\lambda>R/\varepsilon$. In both cases, $\psi$ will change concavity.
\end{proof}
\begin{remark}\label{inflection_pt}
Based on the feedback form \eqref{thetastar}, $\theta^\ast$ is bounded below by $-\frac{\mu-r}{\sigma}$ whenever $\psi$ is convex. If $\psi$ is not convex, then its infection point is the unique point where 
$\theta^\ast=-\frac{\mu-r}{\sigma}$. In other words, $\psi$ changes concavity when the  distorted Sharpe ratio $\frac{\mu-r}{\sigma}+\theta^\ast$ is zero. Moreover, since $\psi$ changes from concave to convex, 
the distorted Sharpe ratio is negative to the left of this inflection point and positive to the right of this inflection point.
\end{remark}

\begin{prop}\label{never_reach_ws}
Let $b<w<w_s$ and $W^\ast$ be the optimally controlled wealth starting at $w$. Let $\tau^\ast_s:=\inf\{t\geq 0: W^\ast_t\geq w_s\}$ and $\tau^\ast_b:=\inf\{t\geq 0: W^\ast_t\leq b\}$. Then $\mathbb{P}(\tau^\ast_s<\tau^\ast_b)=0$.
\end{prop}
\begin{proof}
Since we are only interested in whether the safe level can be reached before ruin, we may extend the domain of $\Pi$ to $\mathbb{R}$ and set $\Pi(w):=\frac{c-rw}{\mu-r}$ for $w\leq b$. Let $\widetilde{W}$ be the solution to the SDE:
\[dW_t=[rW_t+(\mu-r)\Pi(W_t)-c]dt+\sigma \Pi(W_t)dB_t, \quad W_0=w.\]
$\widetilde{W}$ equals $W^\ast$ up to ruin time. It suffices to show $\widetilde{W}$ does not exit the interval $(-\infty, w_s)$ in finite time, and we use Feller's test for explosions (see section 5.5.C of \cite{KS_bible}). By Lemma \ref{pi_bdd}, non-degeneracy and local integrability hold on this interval. Let $\mathfrak{s}(w):=\sigma\Pi(w)$ and $\mathfrak{b}(w):=rw+(\mu-r)\Pi(w)-c$. Fix $w_0\in (-\infty, w_s)$. Let
\[p(w):=\int_{w_0}^w \exp\left(-2\int_{w_0}^y\frac{\mathfrak{b}(z)}{\mathfrak{s}^2(z)}dz\right)dy\]
be the scale function, and
\[v(w):=\int_{w_0}^w p'(y)\int_{w_0}^y\frac{2dz}{p'(z)\mathfrak{s}^2(z)}dy=\int_{w_0}^w \int_{w_0}^y\frac{2}{\mathfrak{s}^2(z)}\exp\left(-2\int_z^y\frac{\mathfrak{b}(x)}{\mathfrak{s}^2(x)}dx\right)dzdy.\]
We want to show $v(-\infty)=v(w_s)=\infty$. $v(-\infty)=\infty$ is easy by the way we extend $\Pi$. Let $a\leq w_0\wedge b$. Since $\mathfrak{b}(x)=0$ for $x\leq a$, we have 
\begin{align*}
v(-\infty)&=\int^{w_0}_{-\infty} \int^{w_0}_y\frac{2}{\mathfrak{s}^2(z)}\exp\left(2\int^z_y\frac{\mathfrak{b}(x)}{\mathfrak{s}^2(x)}dx\right)dzdy\geq \int^{a}_{-\infty} \int^{a}_y\frac{2}{\mathfrak{s}^2(z)}dzdy\\
&=\int^{a}_{-\infty} \int^{a}_y\frac{4R}{(c-rz)^2}dzdy=\infty.
\end{align*}
To show $v(w_s)=\infty$, we use Lemma \ref{pi_bdd} or Proposition \ref{pihat_eps} to obtain $\Pi(w)\leq K_1(c-rw)$, $w\in(b,w_s)$ for some positive constant $K_1$. It follows that $|\mathfrak{b}(w)|\leq [1+K_1(\mu-r)](c-rw)$, $w\in(b,w_s)$. Also observe that if $\mathfrak{b}(w)>0$, then $\Pi(w)>\frac{c-rw}{\mu-r}$. So we have
\[\frac{\mathfrak{b}(w)}{\mathfrak{s}^2(w)}\leq 1_{\{\mathfrak{b}(w)>0\}}\frac{\mathfrak{b}(w)}{\mathfrak{s}^2(w)}\leq1_{\{\mathfrak{b}(w)>0\}}\frac{2R[1+K_1(\mu-r)]}{c-rw}\leq \frac{K_2}{c-rw}, w\in(b,w_s),\]
where $K_2:=2R[1+(\mu-r)K_1]>0$. Let $(b\vee w_0)\leq a'<w_s$.
\begin{align*}
v(w_s)&=\int_{w_0}^{w_s} \int_{w_0}^y\frac{2}{\mathfrak{s}^2(z)}\exp\left(-2\int_z^y\frac{\mathfrak{b}(x)}{\mathfrak{s}^2(x)}dx\right)dzdy\\
&\geq \int_{a'}^{w_s} \int_{a'}^y\frac{2}{\sigma^2 K_1^2(c-rz)^2}\exp\left(-2\int_z^y\frac{K_2}{c-rx}dx\right)dzdy\\
&=\int_{a'}^{w_s} \int_{a'}^y\frac{2}{\sigma^2 K_1^2(c-rz)^2}\left(\frac{c-ry}{c-rz}\right)^{\frac{2K_2}{r}}dzdy\\
%&=\frac{2}{\sigma^2 K^2_1}\int_{a'}^{w_s} (c-ry)^{\frac{2K_2}{r}} \int_{a'}^y\frac{1}{(c-rz)^{2+\frac{2K_2}{r}}}dzdy\\
&=\frac{2}{\sigma^2 K^2_1(r+2K_2)}\int_{a'}^{w_s} (c-ry)^{\frac{2K_2}{r}}\left[(c-ry)^{-1-\frac{2K_2}{r}}-(c-ra')^{-1-\frac{2K_2}{r}}\right]dy\\
&=\frac{2}{\sigma^2 K^2_1(r+2K_2)}\left[\int_{a'}^{w_s} \frac{1}{c-ry}dy-\int_{a'}^{w_s} \frac{1}{c-ra'}\left(\frac{c-ry}{c-ra'}\right)^{\frac{2K_2}{r}}dy\right].
\end{align*}
The second integral is finite while the first integral diverges to $\infty$. So we obtain $v(w_s)=\infty$. The rest  is by Feller's test for explosions.
\end{proof}

\section{Numerical Analysis and Asymptotic Expansion}\label{sec:numerics}

\subsection{Numerical examples}

We solve the boundary value problem \eqref{BVP} numerically using finite difference method. The model parameters used are $c=1$, $b=1$, $\mu=0.1$, $\sigma=0.15$, $\lambda=0.04$ and $\varepsilon=0,1,5,10,50$. We choose a hazard rate of $0.04$, i.e. an expected future lifetime of 25 years, because the investment problem we considered is more relevant to retirees. To demonstrate that the concavity of the value function is closely related to how interest rate compares with hazard rate, we work with two values of interest rate: $r=0.02<\lambda$ and $r=0.06>\lambda$. We plot the robust ruin probability, the optimal investment and the optimal distorted Sharpe ratio as functions of wealth under different levels of ambiguity aversion.  

\begin{figure}[h]
\centering
\includegraphics[height=6cm]{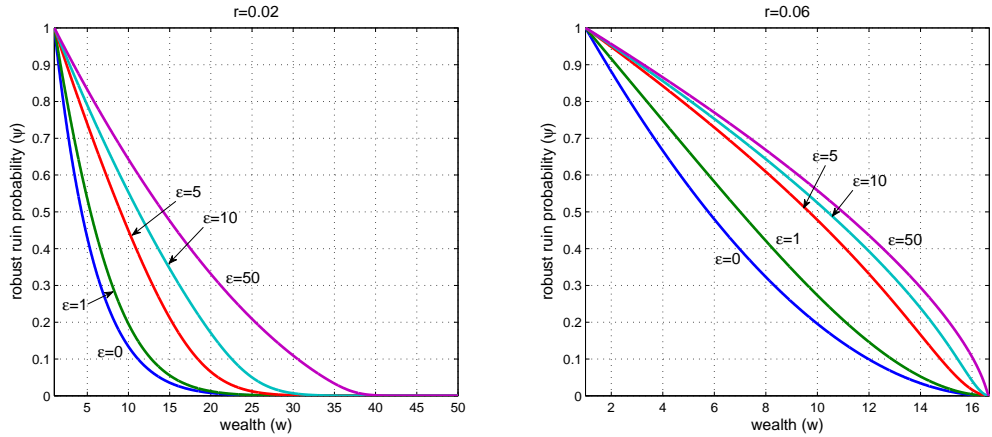}
\caption{Robust ruin probabilities.}
\label{psi}
\end{figure}

\begin{figure}[h]
\centering
\includegraphics[height=4cm]{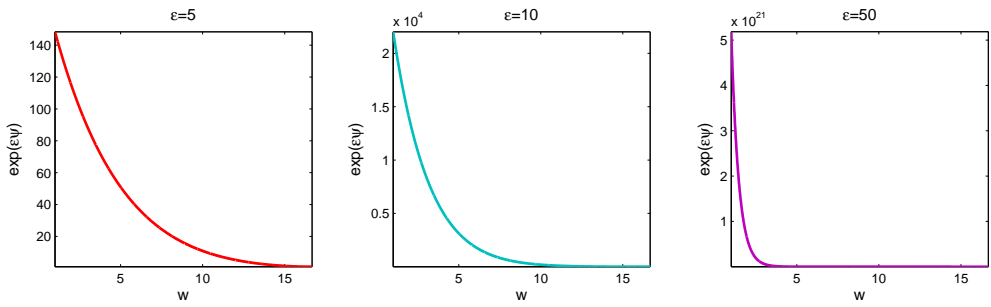}
\caption{Cole-Hopf transform of three non-convex curves in Figure \ref{psi}.}
\label{fig_convex}
\end{figure}

\begin{figure}[h]
\centering
\includegraphics[height=6cm]{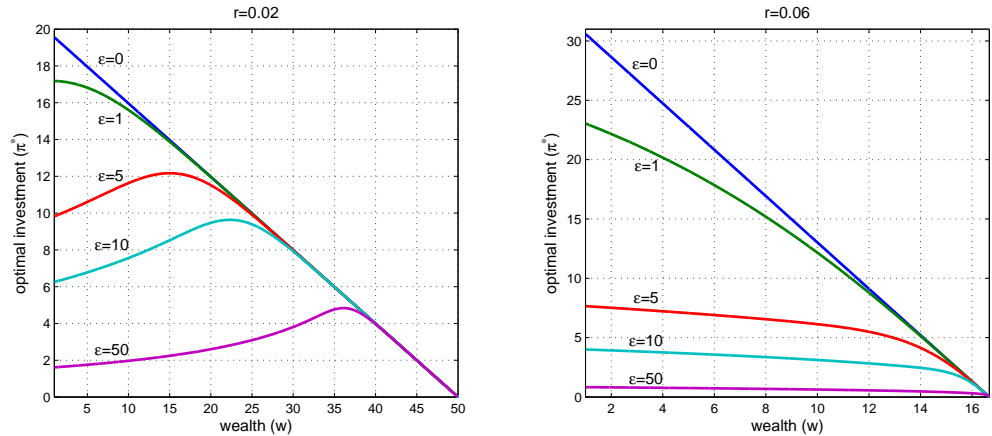}
\caption{Optimal investments.}
\label{pi}
\end{figure}

\begin{figure}[h]
\centering
\includegraphics[height=6cm]{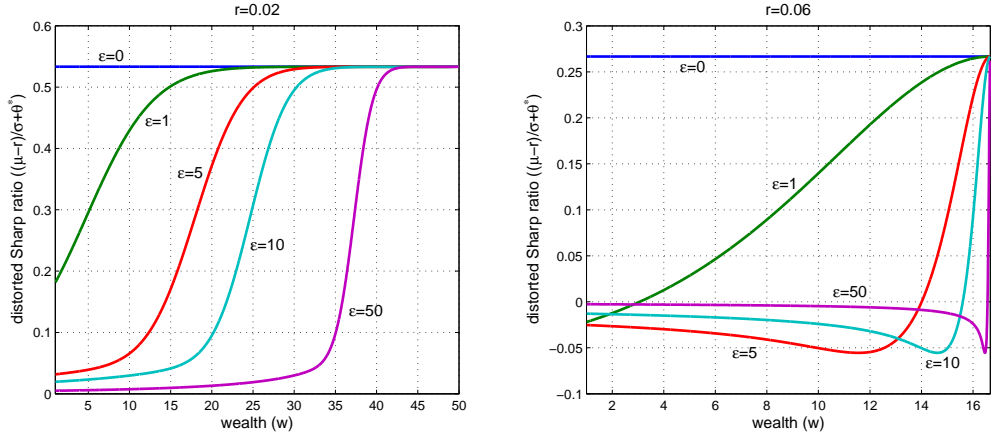}
\caption{Optimally distorted Sharpe ratios.}
\label{SR}
\end{figure}

\begin{figure}[h]
\centering
\includegraphics[height=6cm]{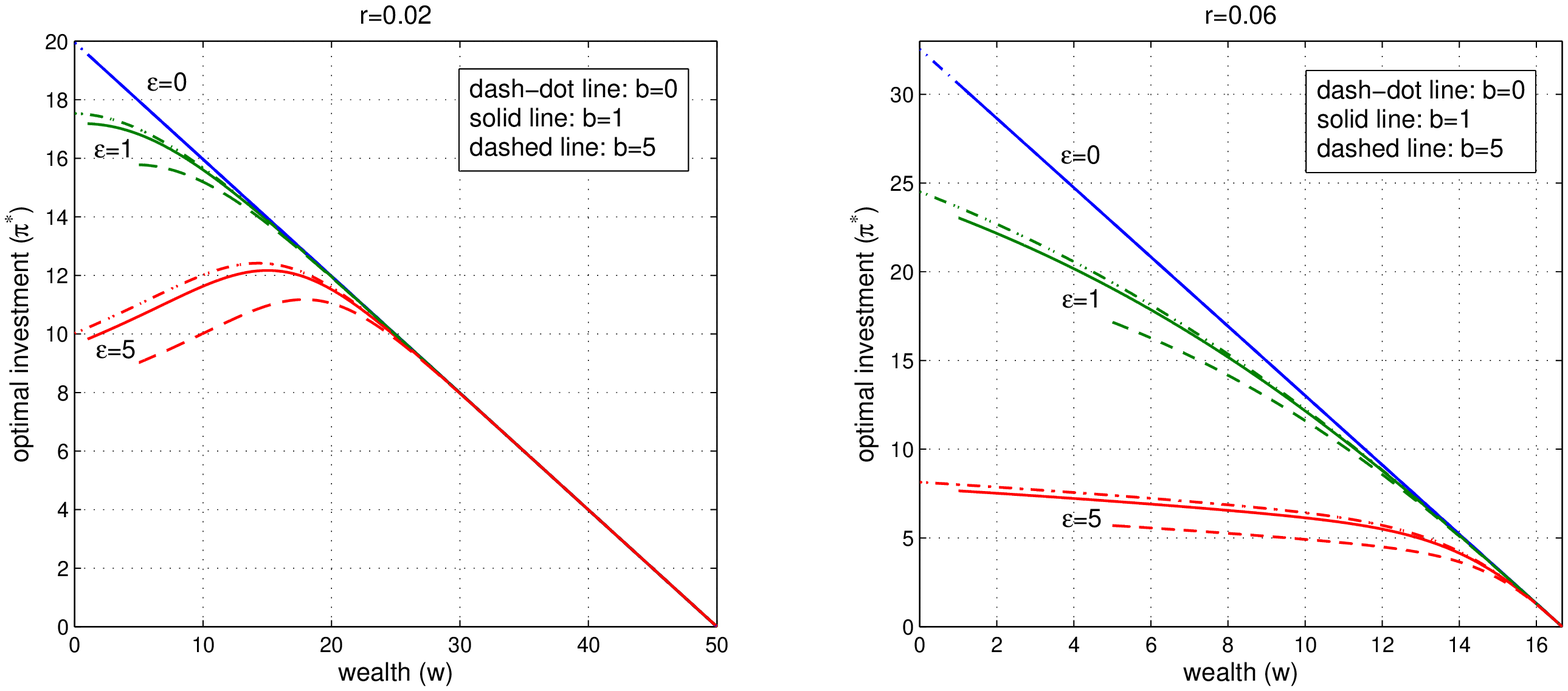}
\caption{Optimal investments with different ruin levels.}
\label{fig_b_dependence}
\end{figure}

From Figure \ref{psi}, we see that the robust value function is increasing in $\varepsilon$. When the interest rate is smaller than the hazard rate, all value functions are strictly convex. When the interest rate is larger than the hazard rate, concavity depends on the level of ambiguity aversion: the value function is convex when $\varepsilon$ is small, and changes from concave to convex when $\varepsilon$ is large. The larger the $\varepsilon$, the closer the inflection point is to the safe level.
With this set of parameters, a sufficient condition for $\psi$ to be convex, as implied by Proposition \ref{psiconvex}, is $0\leq\varepsilon\leq 0.4765$. A sufficient condition for $\psi$ to change concavity is $\varepsilon>1.7778$. $\varepsilon=5$, 10, 50 all satisfy this condition and exhibit concavity change. By Remark \ref{inflection_pt}, the inflection points of $\psi$ corresponds to the points where the optimal distorted Sharpe ratio is zero. Despite that $\psi$ may be concave, its Cole-Hopf transform $e^{\varepsilon \psi}$ is always convex, as demonstrated by Figure \ref{fig_convex}.

Figure \ref{pi} shows the optimal investment level is decreasing in $\varepsilon$, which agrees with Proposition \ref{pihat_eps}. This means the more ambiguity-averse the agent is, the less she is willing to invest in the risky asset. Different from the non-robust case, the optimal investment, although goes to zero as wealth approaches the safe level, is not necessarily a decreasing function of wealth. When interest rate is large (compared with hazard rate), $\pi^\ast$ is decreasing and also concave in $w$. But when interest rate is small, there is an interior point where $\pi^\ast$ achieves maximum. Moreover, as $\varepsilon$ increases, the maximum point moves to the right. In any case, adding ambiguity aversion reduces the amount of borrowing when the wealth of the investor is small, making the model more realistic than the non-robust model (without borrowing constraint). Another interesting observation is that the optimal $\pi^\ast$ of all levels of ambiguity aversion share the same tangent line at the safe level with their non-robust counterparts, confirming equation \eqref{tangent_line_at_ws} of Proposition \ref{pi_Lip}.

Figure \ref{SR} shows that when the interest rate is small (compared with hazard rate), the optimally distorted Sharpe ratio $\frac{\mu-r}{\sigma}+\theta^\ast$ is strictly positive, decreasing in $\varepsilon$ and increasing in wealth. But when interest rate is large, it can be negative, and both monotonicities are lost. In both cases, the pictures suggest that the optimally distorted Sharpe ratio is converging to zero pointwise as $\varepsilon\rightarrow \infty$. Moreover, observe from Figure~\ref{pi} that the optimal investment converges to zero pointwise as $\varepsilon\rightarrow \infty$, which is the investment behavior corresponding to $\psi_{\infty}$. This suggests that for $\varepsilon$ very large, the stock is losing its attractiveness as it becomes less favorable compared to the money market account.

In the non-robust case, the optimal investment strategy is independent of the ruin level $b$ in the sense that if $b_1<b_2$, then $\pi^\ast(w;b_1)$ coincides with $\pi^\ast(w;b_2)$ on $[b_2,w_s]$. This holds not only for constant consumption rate, but also for any Lipschitz continuous consumption rate (see \cite[Corollary 2.3]{BayraktarYoung07a}). However, when ambiguity aversion is present, the ruin level has a global impact on the investment decision unless hazard rate is zero. When $\varepsilon\neq 0$, Figure \ref{fig_b_dependence} suggests $\pi^\ast$ is decreasing in $b$. In other words, the individual will invest less if she is more likely to feel ruined.

\subsection{Asymptotic expansion for small $\varepsilon$}

In general, \eqref{BVP} does not have an explicit solution, but it turns out that for small $\varepsilon$, there are explicit formulas for the leading term and the first order correction. Rewrite \eqref{BVP:a} as
\begin{equation}\label{BVPa2}
\left((rw-c)\psi'-\lambda \psi\right)\left(\varepsilon(\psi')^2+\psi''\right)=R(\psi')^2.
\end{equation}
Let
\begin{equation*}
f_0(w)+f_1(w)\varepsilon+f_2(w)\varepsilon^2+\cdots
\end{equation*}
be an asymptotic expansion of $\psi(w)$ as $\varepsilon\rightarrow 0$. Substituting the expansion into \eqref{BVPa2} and collecting zero-th order terms in $\varepsilon$, we get
\begin{equation}\label{eq_f0}
\left((rw-c)f_0'-\lambda f_0\right)f_0''=R(f_0')^2
\end{equation}
which is precisely the differential equation satisfied by the non-robust value function. We impose the boundary conditions $f_0(b)=1$ and $f_0(w_s)=0$. Then
\begin{equation*}
f_0(w)=\psi_0(w)=\left(\frac{c-rw}{c-rb}\right)^d.
\end{equation*}
Collecting first order terms in $\varepsilon$, we get
\begin{equation}\label{eq_f1}
[(rw-c)f'_1-\lambda f_1]f''_0+[(rw-c)f'_0-\lambda f_0][(f'_0)^2+f''_1]=2Rf'_0f'_1.
\end{equation}
Using the formula for $f_0$, after some computation, we arrive at a linear second order ODE for $f_1$:
\begin{equation}\label{f1}
f''_1+A(w)f_1'+B(w)f_1+C(w)=0
\end{equation}
where
\begin{align*}
A(w)&:=\frac{r(d-1)(2R-rd+r)}{R}\frac{1}{c-rw},\\
B(w)&:=\frac{-\lambda r^2 (d-1)^2}{R}\frac{1}{(c-rw)^2},\\
C(w)&:=\frac{r^2 d^2}{(c-rb)^{2d}}(c-rw)^{2d-2}.
\end{align*}
We require $f_1$ to satisfy the homogeneous boundary conditions $f_1(b)=f_1(w_s)=0$. Let $x=c-rw$ and $g(x)=f_1(w)$. Equation \eqref{f1} can be rewritten as
\begin{equation}\label{ODEg}
x^2g''-\frac{(d-1)(2R-rd+r)}{R}xg'-\frac{\lambda(d-1)^2}{R}g+\frac{d^2}{(c-rb)^{2d}}x^{2d}=0
\end{equation}
with boundary conditions $g(0)=g(c-rb)=0$. This is a non-homogeneous Cauchy-Euler equation. The corresponding homogeneous equation has general solution:
\[g_h(x)=C_1x^{k_1}+C_2x^{k_2}\]
where $k_1>0>k_2$ are the roots of
\[k^2-\left(2d-1-\frac{r(d-1)^2}{R}\right)k-\frac{\lambda(d-1)^2}{R}=0.\]
It turns out that $k_1=d$. For a particular solution, we guess the form $g_p(x)=C_px^{2d}$. Substituting $g_p$ into \eqref{ODEg}, we find
\[C_p=\frac{-Rd^2}{(c-rb)^{2d}\left[(d-1)^2(2dr-\lambda)+2Rd\right]}\]
The general solution to \eqref{ODEg} is $g=g_h+g_p$. Since $g(0)=0$, we must have $C_2=0$, otherwise the solution would explode at $x=0$. The other boundary condition $g(c-rb)=0$ yields
\[C_1=-C_p(c-rb)^{d}.\]
So we have obtained
\begin{equation*}
f_1(w)=g(c-rw)=\frac{Rd^2}{(d-1)^2(2dr-\lambda)+2Rd}\left[\left(\frac{c-rw}{c-rb}\right)^{d}-\left(\frac{c-rw}{c-rb}\right)^{2d}\right].
\end{equation*}

\begin{prop}
\[\psi(w)=\psi_0(w)+\frac{Rd^2\left(\psi_0(w)-\psi_0^{2}(w)\right)}{(d-1)^2(2dr-\lambda)+2Rd}\,\varepsilon+O(\varepsilon^2)\]
as $\varepsilon\downarrow 0$ uniformly in $w$, where the constants $R,d$ are defined in \eqref{Rd}.
\end{prop}
\begin{proof}
We only give a sketch proof. Let $\tilde{\psi}:=f_0+f_1\varepsilon$. We want to show $\psi(w)=\tilde{\psi}(w)+O(\varepsilon^2)$ uniformly for $w\in(b,w_s)$. Using the formulas for $f_0$ and $f_1$, we can show for $\varepsilon$ sufficiently small,
\begin{equation}\label{approx_eq1}
\varepsilon (\tilde{\psi}')^2(w)+\tilde{\psi}''(w)\geq C_1\left(\frac{c-rw}{c-rb}\right)^{d-2}>0 \quad \forall \ w\in(b,w_s).
\end{equation}
for some positive constant $C_1$ independent of $\varepsilon$ and $w$. Next, we show
\begin{equation}\label{approx_eq2}
F(w,\tilde{\psi}(w),\tilde{\psi}'(w),\tilde{\psi}''(w))=O(\varepsilon^2) \quad\text{uniformly for }w\in(b,w_s).
\end{equation}
In view of \eqref{approx_eq1}, we carry out the optimization over $\pi$ in the expression for $F$. Using equations \eqref{eq_f0} and \eqref{eq_f1}, we obtain
\begin{equation*}
F(w,\tilde{\psi}(w),\tilde{\psi}'(w),\tilde{\psi}''(w))=\frac{D_1(w)\varepsilon^2+D_2(w)\varepsilon^3+D_3(w)\varepsilon^4}{\varepsilon (\tilde{\psi}')^2(w)+\tilde{\psi}''(w)},
\end{equation*}
where
\begin{align*}
D_1(w)&:=2f'_0f'_1[\lambda f_0-(rw-c)f'_0]+[\lambda f_1-(rw-c)f'_1][(f'_0)^2+f''_1]+R(f'_1)^2,\\
D_2(w)&:=2f'_0f'_1[\lambda f_1-(rw-c)f'_1]+(f'_1)^2[\lambda f_0-(rw-c)f'_0],\\
D_3(w)&:=(f'_1)^2[\lambda f_1-(rw-c)f'_1].
\end{align*}
It can be shown using \eqref{approx_eq1} and the explicit formulas for $f_0$, $f_1$ that there exists a positive constant $C_2$ independent of $\varepsilon$ and $w$, such that $|D_i(w)|/[\varepsilon (\tilde{\psi}')^2(w)+\tilde{\psi}''(w)]\leq C_2, i=1,2,3$ for all $w\in(b,w_s)$ and for $\varepsilon$ small enough. This proves \eqref{approx_eq2}. Consequently, we can find a positive constant $C_3$ such that for $\varepsilon$ sufficiently small,
\[F(w,\tilde{\psi}(w)-C_3\varepsilon^2,\tilde{\psi}'(w),\tilde{\psi}''(w))\leq 0 \ \text{ and } \ F(w,\tilde{\psi}(w)+C_3\varepsilon^2,\tilde{\psi}'(w),\tilde{\psi}''(w))\geq 0.\]
By comparison principle for the equation $F=0$, we have $\psi=\tilde{\psi}+O(\varepsilon^2)$.
\end{proof}

\subsection{Should the individual care about robustness?}
Figure \ref{psi} shows robustness has a considerable impact on the minimum probability of ruin. However, this is not really informative as far as investment behavior is concerned. A more important question is: how does the optimal non-robust investment strategy $\pi_0$ perform in the robust market? In other words, will the individual bear significantly more risk if she makes investment decisions as if there were no model uncertainty? The answer to this question is partially affirmative; the individual should care about robustness for non-small $\varepsilon$. In our numerical example, ignoring robustness increases the ruin probability by more than 10\% for $\varepsilon$ larger than 10. On the other hand, for small $\varepsilon$, $\pi_0$ turns out to be a good enough investment strategy. For $\varepsilon=1$, the difference between the ruin probability yielded by $\pi_0$ and the optimal ruin probability $\psi(\cdot \,;1)$ is on a scale of $0.1\%$ which may be negligible for an individual, although the difference between $\psi_0$ and $\psi(\cdot \,;1)$ can be as large as 10\%. Table \ref{tab:table} illustrates  the performance of $\pi_0$ under various levels of ambiguity aversion  in our numerical example.

\newcolumntype{Z}{>{\centering\arraybackslash }X}
\begin{center}
\begin{table}[h]
\caption{Maximum deviation from the minimum robust ruin probability if the individual uses the non-robust strategy $\pi_0$.\label{tab:table}}
\begin{tabularx}{0.8\textwidth}{ |c|Z|Z|Z|Z|Z|Z|Z|}
  \hline
  &$\varepsilon=1$ & $\varepsilon=2$ & $\varepsilon=3$ & $\varepsilon=4$ & $\varepsilon=5$ &  $\varepsilon=10$ & $\varepsilon=20$ \\
  \hline
 $r=0.02$   & 0.001  & 0.005 & 0.013  &  0.025 &  0.038 & 0.105 & 0.201 \\
  \hline
  $r=0.06$  &  0.002 & 0.013 & 0.033  & 0.059 & 0.087 & 0.198 & 0.324 \\ 
  \hline
\end{tabularx}
\end{table}
\end{center}

\pagebreak

\bibliographystyle{plain}
\bibliography{minimumWealth}{}

\def\cprime{$'$} \def\cprime{$'$}
\begin{thebibliography}{10}

\bibitem{Alexandroff}
A.~D. Alexandroff.
\newblock Almost everywhere existence of the second differential of a convex
  function and some properties of convex surfaces connected with it.
\newblock {\em Leningrad State Univ. Annals [Uchenye Zapiski] Math. Ser.},
  6:3--35, 1939.

\bibitem{Lion97}
O.~Alvarez, J.-M. Lasry, and P.-L. Lions.
\newblock Convex viscosity solutions and state constraints.
\newblock {\em J. Math. Pures Appl. (9)}, 76(3):265--288, 1997.

\bibitem{Bauerle13}
Nicole B{\"a}uerle and Erhan Bayraktar.
\newblock A note on applications of stochastic ordering to control problems in
  insurance and finance.
\newblock {\em Stochastics}, 86(2):330--340, 2014.

\bibitem{Xueying11}
E.~Bayraktar, X.~Hu, and V.~R. Young.
\newblock Minimizing the probability of lifetime ruin under stochastic
  volatility.
\newblock {\em Insurance Math. Econom.}, 49(2):194--206, 2011.

\bibitem{BayraktarSong}
E.~Bayraktar and S.~Yao.
\newblock A weak dynamic programming principle for zero-sum stochastic
  differential games with unbounded controls.
\newblock {\em SIAM J. Control Optim.}, 51(3):2036--2080, 2013.

\bibitem{BayraktarYoung07a}
E.~Bayraktar and V.~R. Young.
\newblock {Correspondence between lifetime minimum wealth and utility of
  consumption.}
\newblock {\em {Finance Stoch.}}, 11(2):213--236, 2007.

\bibitem{BayraktarYoung07b}
E.~Bayraktar and V.~R. Young.
\newblock Minimizing the probability of lifetime ruin under borrowing
  constraints.
\newblock {\em Insurance Math. Econom.}, 41(1):196--221, 2007.

\bibitem{BayraktarYoung08}
E.~Bayraktar and V.~R. Young.
\newblock Minimizing the probability of ruin when consumption is ratcheted.
\newblock {\em N. Am. Actuar. J.}, 12(4):428--442, 2008.

\bibitem{BayraktarYoung11}
E.~Bayraktar and V.~R. Young.
\newblock Proving regularity of the minimal probability of ruin via a game of
  stopping and control.
\newblock {\em Finance Stoch.}, 15(4):785--818, 2011.

\bibitem{BenBertBrown}
A.~Ben-Tal, D.~Bertsimas, and D.~B. Brown.
\newblock A soft robust model for optimization under ambiguity.
\newblock {\em Oper. Res.}, 58(4, part 2):1220--1234, 2010.

\bibitem{Bordigoni07}
G.~Bordigoni, A.~Matoussi, and M.~Schweizer.
\newblock A stochastic control approach to a robust utility maximization
  problem.
\newblock In {\em Stochastic analysis and applications}, volume~2 of {\em Abel
  Symp.}, pages 125--151. Springer, Berlin, 2007.

\bibitem{Jaimungal13}
A.~Cartea, R.~Donnelly, and S.~Jaimungal.
\newblock Robust market making.
\newblock 2013.
\newblock Available at SSRN: http://ssrn.com/abstract=2310645.

\bibitem{UsersGuide}
M.~G. Crandall, H.~Ishii, and P.-L. Lions.
\newblock {User's guide to viscosity solutions of second order partial
  differential equations.}
\newblock {\em {Bull. Am. Math. Soc., New Ser.}}, 27(1):1--67, 1992.

\bibitem{Duffie97}
D.~Duffie, W.~H. Fleming, H.~M. Soner, and T.~Zariphopoulou.
\newblock Hedging in incomplete markets with {HARA} utility.
\newblock {\em J. Econom. Dynam. Control}, 21(4-5):753--782, 1997.

\bibitem{Fleming11}
W.~H. Fleming and D.~Hern{\'a}ndez-Hern{\'a}ndez.
\newblock On the value of stochastic differential games.
\newblock {\em Commun. Stoch. Anal.}, 5(2):341--351, 2011.

\bibitem{Fleming89}
W.~H. Fleming and P.~E. Souganidis.
\newblock On the existence of value functions of two-player, zero-sum
  stochastic differential games.
\newblock {\em Indiana Univ. Math. J.}, 38(2):293--314, 1989.

\bibitem{Hansen}
L.~P. Hansen and T.~J. Sargent.
\newblock {\em Robustness}.
\newblock Princeton University Press, 2007.

\bibitem{Hansen06}
L.~P. Hansen, T.~J. Sargent, G.~Turmuhambetova, and N.~Williams.
\newblock Robust control and model misspecification.
\newblock {\em J. Econom. Theory}, 128(1):45--90, 2006.

\bibitem{HSchied07}
D.~Hern{\'a}ndez-Hern{\'a}ndez and A.~Schied.
\newblock A control approach to robust utility maximization with logarithmic
  utility and time-consistent penalties.
\newblock {\em Stochastic Process. Appl.}, 117(8):980--1000, 2007.

\bibitem{Hu11}
Y.~Hu and M.~Schweizer.
\newblock Some new {BSDE} results for an infinite-horizon stochastic control
  problem.
\newblock In {\em Advanced mathematical methods for finance}, pages 367--395.
  Springer, Heidelberg, 2011.

\bibitem{Huang92}
C.~Huang and H.~Pag{\`e}s.
\newblock Optimal consumption and portfolio policies with an infinite horizon:
  existence and convergence.
\newblock {\em Ann. Appl. Probab.}, 2(1):36--64, 1992.

\bibitem{Jacka02}
S.~Jacka.
\newblock Avoiding the origin: a finite-fuel stochastic control problem.
\newblock {\em Ann. Appl. Probab.}, 12(4):1378--1389, 2002.

\bibitem{Jacod}
J.~Jacod and A.~N. Shiryaev.
\newblock {\em Limit Theorems for Stochastic Processes}.
\newblock Srpinger, 2nd edition, 2002.

\bibitem{Jaimungal11}
S.~Jaimungal.
\newblock Irreversible investments and ambiguity aversion.
\newblock 2011.
\newblock Available at SSRN: http://ssrn.com/abstract=1961786.

\bibitem{Jaimungal12}
S.~Jaimungal and G.~Sigloch.
\newblock Incorporating risk and ambiguity aversion into a hybrid model of
  default.
\newblock {\em Math. Finance}, 22(1):57--81, 2012.

\bibitem{Janecek12}
K.~Jane{\v{c}}ek and M.~S{\^{\i}}rbu.
\newblock Optimal investment with high-watermark performance fee.
\newblock {\em SIAM J. Control Optim.}, 50(2):790--819, 2012.

\bibitem{KS_bible}
I.~Karatzas and S.~E. Shreve.
\newblock {\em Brownian Motion and Stochastic Calculus}.
\newblock Springer, 2nd edition, 1991.

\bibitem{Krylov}
N.~V. Krylov.
\newblock {\em Nonlinear Elliptic and Parabolic Equations of the Second Order}.
\newblock D. Reidel Publishing Company, Dordrecht, Holland, 1987.

\bibitem{LimShanthikumar}
A.~E.~B. Lim and J.~G. Shanthikumar.
\newblock Relative entropy, exponential utility, and robust dynamic pricing.
\newblock {\em Oper. Res.}, 55(2):198--214, 2007.

\bibitem{Maenhout04}
P.~Maenhout.
\newblock Robust portfolio rules and asset pricing.
\newblock {\em The Review of Financial Studies}, 17(4):951--983, 2004.

\bibitem{MataramvuraOksendal}
S.~Mataramvura and B.~{\O}ksendal.
\newblock {Risk minimizing portfolios and HJBI equations for stochastic
  differential games.}
\newblock {\em {Stochastics}}, 80(4):317--337, 2008.

\bibitem{MilevskyRobinson00}
M.~A. Milevsky and C.~Robinson.
\newblock Self-annuitization and ruin in retirement.
\newblock {\em N. Am. Actuar. J.}, 4(4):112--129, 2000.
\newblock With discussion.

\bibitem{Sudderth85}
V.~C. Pestien and W.~D. Sudderth.
\newblock Continuous-time red and black: how to control a diffusion to a goal.
\newblock {\em Math. Oper. Res.}, 10(4):599--611, 1985.

\bibitem{Rogers}
L.~C.~G. Rogers.
\newblock {\em Optimal Investment}.
\newblock Springer Berlin Heidelberg, 2013.

\bibitem{Schied07}
A.~Schied.
\newblock Optimal investments for risk- and ambiguity-averse preferences: a
  duality approach.
\newblock {\em Finance Stoch.}, 11(1):107--129, 2007.

\bibitem{ShreveSoner94}
S.~E. Shreve and H.~M. Soner.
\newblock Optimal investment and consumption with transaction costs.
\newblock {\em Ann. Appl. Probab.}, 4(3):609--692, 1994.

\bibitem{Sirbu13}
Mihai {S\^{\i}rbu}.
\newblock {Stochastic Perron's method and elementary strategies for zero-sum
  differential games.}
\newblock {\em {SIAM J. Control Optim.}}, 52(3):1693--1711, 2014.

\bibitem{Stroock}
D.~W. Stroock.
\newblock {\em Lectures on Stochastic Analysis: Diffusion Theory}.
\newblock Cambridge University Press, 1987.

\bibitem{Taylor52}
A.~E. Taylor.
\newblock L'{H}ospital's rule.
\newblock {\em Amer. Math. Monthly}, 59:20--24, 1952.

\bibitem{Touzi}
N.~Touzi.
\newblock {\em Optimal Stochastic Control, Stochastic Target Problems, and
  Backward SDE}.
\newblock Springer New York Heidelberg Dordrecht London, 2012.

\bibitem{Wang03}
R.~Uppal and T.~Wang.
\newblock Model misspecification and underdiversification.
\newblock {\em The Journal of Finance}, 58(6):2465--2486, 2003.

\bibitem{Young04}
V.~R. Young.
\newblock {Optimal investment strategy to minimize the probability of lifetime
  ruin.}
\newblock {\em {N. Am. Actuar. J.}}, 8(4):105--126, 2004.

\bibitem{Zariphopoulou94}
T.~Zariphopoulou.
\newblock Consumption-investment models with constraints.
\newblock {\em SIAM J. Control Optim.}, 32(1):59--85, 1994.

\end{thebibliography}

\end{document}